\documentclass[journal,twoside,web]{ieeecolor}

\usepackage{generic}
\usepackage{cite}
\usepackage{graphicx}
\usepackage{amsmath} 
\usepackage{amssymb}  
\usepackage{lipsum}
\usepackage{amsfonts}
\usepackage{dblfloatfix}
\usepackage{mathrsfs}
\usepackage{mathtools}
\usepackage[thmmarks, amsmath]{ntheorem}
\usepackage[bbgreekl]{mathbbol}
\usepackage{multirow}
\usepackage[caption=false]{subfig}
\usepackage{makecell}

\usepackage{booktabs}
\usepackage{pdfcomment}
\hypersetup{hidelinks}

\usepackage{cuted}

\usepackage{algpseudocode}
\usepackage{algorithm}
\DeclareFontEncoding{LS1}{}{}
\DeclareFontSubstitution{LS1}{stix}{m}{n}
\DeclareSymbolFont{stixsymbols}       {LS1}{stixscr}  {m} {n}
\DeclareSymbolFontAlphabet{\mathscrl} {stixsymbols}

\algnewcommand{\algorithmicand}{\textbf{ and }}
\algnewcommand{\algorithmicor}{\textbf{ or }}
\algnewcommand{\OR}{\algorithmicor}
\algnewcommand{\AND}{\algorithmicand}
\algnewcommand{\TRUE}{\textbf{ true }}
\algnewcommand{\FALSE}{\textbf{ false }}
\algnewcommand{\NOT}{\textbf{ not }}
\algnewcommand{\BREAK}{\textbf{ break }}
\algrenewcommand\alglinenumber[1]{{\sffamily\footnotesize#1}}

\algnewcommand{\Initialize}[1]{%
	\State \textbf{Initialize:}
	\Statex \hspace*{\algorithmicindent}\parbox[t]{.8\linewidth}{\raggedright #1}
}

\DeclareSymbolFontAlphabet{\mathbb}{AMSb}
\DeclareSymbolFontAlphabet{\mathbbl}{bbold}

\DeclareMathOperator{\diag}{diag}
\DeclareMathOperator{\rank}{rank}
\DeclareMathOperator{\blkdiag}{blkdiag}

\DeclareMathOperator{\diff}{d}

\DeclareMathOperator{\Null}{null}

\newcommand{\R}{{\mathbb R}}
\newcommand{\N}{{\mathbb N}}
\newcommand{\mc}{\mathcal}
\newcommand{\ddt}{\tfrac{\diff}{\diff \!t}}
\newcommand{\ddtn}[1]{\tfrac{\diff^{#1}}{\diff \!t^{#1}}}

\newcommand{\dc}{{\text{dc}}}
\newcommand{\ac}{{\text{ac}}}
\newcommand{\g}{{\text{g}}}

\newcommand{\el}{{\text{l}}}
\newcommand{\oo}{{\text{o}}}

\newcommand{\cc}{{\text{c}}}

\newcommand{\T}{{\mathsf{T}}}

\newtheorem{theorem}{Theorem}
\newtheorem{lemma}{Lemma}
\newtheorem{proposition}{Proposition}
\newtheorem{assumption}{Assumption}

\newtheorem{definition}{Definition}
\newtheorem{condition}{Condition}
\newtheorem{example}{Example}
\newtheorem{corollary}{Corollary}

\newcommand{\acdc}{{\text{c}}}

\newcommand{\ad}{{{\text{c}_\text{ac}}}}
\newcommand{\da}{{{\text{c}_\text{dc}}}}

\title{\LARGE \bf  Power-balancing dual-port grid-forming power converter control for renewable integration and hybrid AC/DC power systems\thanks{This work was partially funded by the Swiss Federal Office of Energy under grant number SI/501707.}}

\author{Irina Suboti\'c and Dominic Gro\ss{} \thanks{I. Suboti\'c is with the Automatic Control Laboratory at ETH Z\"urich, Switzerland, D. Gro\ss{} is with the Department of Electrical and Computer Engineering at the University of Wisconsin-Madison, USA; e-mail:subotici@ethz.ch,  dominic.gross@wisc.edu}}

\begin{document}
\maketitle
\begin{abstract} 
In this work, we investigate grid-forming (GFM) control for dc/ac power converters in emerging power systems that contain ac and dc networks, renewable generation, and conventional generation. We propose a change in control paradigm to a universal dual-port GFM control strategy that simultaneously forms the converter ac and dc voltage (i.e., dual-port GFM), unifies standard functions of grid-following (GFL) and GFM (e.g., primary frequency control, maximum power point tracking) in a single universal controller, and is backwards compatible with conventional machine-based generation. Notably, in contrast to state-of-the-art control architectures that use a mix of grid-forming and grid-following control, dual-port GFM control can be used independently of the converter power source or network configuration. Our main contribution are stability conditions that cover emerging hybrid ac/dc networks as well as machines and converters with and without controlled power source, that only require partial knowledge of the network topology. Finally, a detailed case study is used to illustrate and validate the results.
\end{abstract}
\section{Introduction}
A major transition in the operation of electric power systems is the replacement of conventional fuel-based power generation interfaced via synchronous machines by distributed renewable generation interfaced via dc/ac power converters. In contrast to machine-interfaced conventional generation that stabilize power systems through their physical properties (e.g., rotational inertia) and controls (e.g., speed governor), today's converter-interfaced renewables are controlled to maximize energy yield and can jeopardize system reliability \cite{WEB+15,MDH+18}.

Control strategies for grid-connected dc/ac voltage source converters (VSC) are typically categorized into (i) grid-forming (GFM) strategies that impose a stable ac voltage waveform (e.g., frequency and magnitude) at the point of connection, and (ii) grid-following (GFL) controls that require stabilization of the ac voltage waveform (e.g., frequency and magnitude) at their point of interconnection by other devices in the grid (e.g., machines).

While this classification commonly refers to the converter ac terminal (i.e., ac-GFM and ac-GFL), it is also useful to characterize the dc terminal, i.e., dc-GFM strategies control the dc terminal voltage and dc-GFL strategies require stabilization of the dc terminal voltage by other devices connected to the converter dc terminal.

Conceptually, standard ac-GFM/dc-GFL strategies impose an ac terminal voltage with variable frequency and magnitude that is adjusted in response to deviations of the ac side active and reactive power from their setpoints. The prevalent ac-GFM/dc-GFL strategies in the literature are droop-control \cite{CDA93}, synchronous machine emulation \cite{BIY14,DSF2015}, and (dispatchtable) virtual oscillator control \cite{JDH+14,GCB+19,SG+20}, that provide fast and reliable grid-support and are envisioned to replace synchronous machines as the cornerstone of future large-scale ac power systems \cite{MDH+18,LCP20}. Moreover, ac-GFM/dc-GFL strategies are used in high-voltage dc (HVDC) applications (e.g., to impose a stable ac voltage in offshore ac networks). In contrast, ac-GFL/dc-GFM strategies control the ac current to stabilize the dc voltage. The prevalent ac-GFL/dc-GFM strategies in the literature rely on a phase-locked loop (PLL) and are used to interface renewable generation with limited or no controllability (e.g., control the dc voltage of a solar photovoltaic (PV) system operating at its maximum power point (MPP) \cite{HS+17}) and HVDC (e.g., control the dc terminal voltage based on the deviation of the dc power from its setpoint \cite{GO+21}).

The two approaches are complementary in the sense that ac-GFM/dc-GFL requires a stable dc voltage and controls the ac voltage, while ac-GFL/dc-GFM require a stable ac voltage and control the dc voltage. While ac-GFL control cannot operate without sufficient ac-GFM resources \cite{MSV+19} and is vulnerable to grid disturbances \cite{NERC17}, ac-GFM/dc-GFL control fails if the converter dc voltage is not tightly controlled by a power source \cite{TGA+20}. Thus, at present, a mix of ac-GFL/dc-GFM and ac-GFM/dc-GFL control is needed to operate emerging power systems that contain renewable generation as well as a mix of high-voltage ac and HVDC transmission \cite{GO+21}. The resulting complex heterogeneous system dynamics pose significant challenges for system operation and stability analysis. 

The literature on stability analysis of ac power systems with power converters can be categorized into works using analytical \cite{GCB+19,SG+20,BHO2020,CGD17,MDPS+17,TAD2020} and numerical \cite{PPG2007,LD14,RMK2015,MSV+19,VHA+2018,TGA+20,LCP20} approaches. Most  of these works only consider ac-GFM converters with the dc terminal modeled as constant voltage source \cite{PPG07,BHO2020,LD14,VHA+2018,SG+20} and only few numerical works \cite{LCP20} consider renewable generation with limited controllability. Moreover, none of these works \cite{PPG07,BHO2020,LD14,VHA+2018,SG+20,LCP20} consider synchronous machines, synchronous condensers, GFL converters, or dc transmission. Some numerical studies consider  ac-GFM droop and PLL-based ac-GFL \cite{RMK2015} or ac-GFM droop, PLL-based ac-GFL, and synchronous machines \cite{MSV+19}, but model the dc terminal as constant voltage source. In contrast, \cite{TGA+20,CGD17,MDPS+17,LCP20,TAD2020} assumed that every converter is connected to a dc power source that implements proportional dc voltage control.
Overall, to the best of our knowledge, no analytical stability results are available in the literature that cover power systems containing both ac and dc networks,  synchronous machines and synchronous condensers, power generation with limited controllability, and power converters providing common ac-GFM (e.g., primary frequency control) and ac-GFL functions (e.g., maximum power point tracking).

Our contribution is two fold. First, we propose the concept of dual-port GFM control that subsumes the functions provided by standard ac-GFM and ac-GFL controls in a simple single \emph{universal control} strategy that significantly reduces system complexity. Second, we leverage the properties of dual-port GFM control to develop analytical stability conditions for linear reduced-order models of power systems with ac and dc transmission, renewable generation either at a curtailed operating point or at its MPP, and conventional generation.

Specifically, dual-port GFM control imposes the ac frequency through active power droop \cite{CDA93} and dc voltage droop \cite{CBB+2015,MDPS+17,AJD18} while ensuring power balancing between the ac and dc terminal by mapping the signals indicating power imbalance (i.e., the ac frequency and dc voltage deviation) between the converter ac and dc terminals (see Fig.~\ref{fig:imbalance}).

This \emph{universal control} (i) can be applied to all  aforementioned technologies, (ii) significantly reduces the complexity of the overall system dynamics, and (iii) enables analytical stability conditions for power systems containing a wide range of legacy technologies (e.g., synchronous machines and synchronous condensers) and emerging technologies (e.g., PV, wind power, and HVDC).
\begin{figure}[h]
	\centering
	\includegraphics[width=0.43\textwidth]{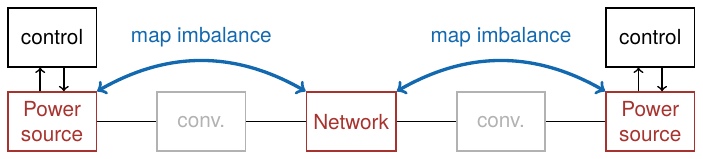}
	\caption{Dual-port grid forming control propagates power imbalances.
		\label{fig:imbalance}}
\end{figure}

First, in Sec.~\ref{sec:power.system.model}, we develop a graph representation of complex power systems with ac and dc networks and present a tractable reduced-order linearized models of converters, machines, and power sources that, in abstraction, model devices ranging from machines with turbine governor system to synchronous condensers, solar PV systems, and HVDC converters. Based on this model, we propose a novel unified power-balancing dual-port GFM controller that simultaneously imposes the converter ac voltage and controls its dc voltage while ensuring power balancing between the ac and dc terminals. We also illustrate that, in this general setup, stability conditions that are independent of the network topology can no longer be obtained.

Next, in Sec.~\ref{sec:stability}, we develop our main theoretical contribution and provide conditions for ac frequency / dc voltage stability of hybrid ac/dc power systems that account for devices (e.g., synchronous condensers, HVDC converters, renewables operating approximately at MPPT) that are not directly connected to a stabilizing power source (e.g., turbine-governor system or dc power source providing dc voltage control) and only require partial knowledge of the system topology.  Finally, in Sec.~\ref{sec:test.case.study} electromagnetic transient (EMT) simulations of a detailed case study are used to illustrate and validate the results, and Sec.~\ref{sec:conclusion} concludes the paper.

\subsection*{Notation}
We use $\R$ and $\mathbb N$ to denote the set of real and natural numbers and define $\R_{\geq a}\coloneqq \{x \in \mathbb R \vert x \geq a\}$ and, e.g., $\R_{[a,b)}\coloneqq \{x \in \mathbb R \vert a \leq x < b\}$. Given a matrix $A$, $A^\mathsf{T}$ denotes its transpose. We write $A\succcurlyeq0$ ($A\succ0$) to denote that $A$ is symmetric and positive semidefinite (definite). For column vectors $x\in\R^n$ and $y\in\R^m$ we use $(x,y) = [x^\mathsf{T}, y^\mathsf{T}]^\mathsf{T} \in \R^{n+m}$ to denote a stacked vector.
Furthermore, $I_n$ denotes the identity matrix of dimension $n$, matrices of zeros of dimension $n \times m$ are denoted by $\mathbbl{0}_{n\times m}$, and $\mathbbl{0}_{n}$ and $\mathbbl{1}_{n}$ denote column vector of zeros and ones of length $n$. If a matrix $M$ does not contain a row or column, we call it an empty matrix.  The cardinality of a set $\mc X \subset \mathbb{N}$ is denoted by $|\mc X|$.
 	
\section{Power system model} \label{sec:power.system.model}
In this section, we introduce a tractable reduced-order order model of a power system containing ac and dc transmission, power converters, and machines. Based on this model, we propose a novel unified dual-port GFM controller. 
\subsection{Hybrid DC/AC power network topology} \label{sec:power.network.topology}
The network is modeled as a connected, undirected, simple graph $\mc G_N=(\mc N_N, \mc E_N)$, where $\mc N_N \coloneqq \mc N_\ac \cup \mc N_\dc \cup \mc N_\cc$  consists of ac nodes $\mc N_{\ac}$ corresponding to machines, dc/ac nodes $\mc N_\cc$ corresponding to power converters, and dc buses $\mc N_\dc$. We distinguish two types of edges: ac edges $\mc E_\ac \in (\mc N_\ac \cup \mc N_\cc) \times (\mc N_\ac \cup \mc N_\cc)$ corresponding to ac connections, and dc edges $\mc E_\dc \in \mc (\mc N_\dc \cup \mc N_\cc) \times (\mc N_\dc \cup \mc N_\cc)$ corresponding to dc connections. Fig.~\ref{fig:gt} shows an example of a hybrid ac/dc network. Note that the ac and dc edges do not necessarily correspond to transmission lines, but generic ac and dc connections between converters and machines (cf. Fig.~\ref{fig:b2bwttop}). 

Next, we partition the network into an ac network $\mc G_\ac=(\mc N_\ac \cup \mc N_\cc,\mc E_\ac)$ (red in Fig.~\ref{fig:gt}), and a dc network $\mc G_\dc=(\mc N_\dc \cup \mc N_\cc,\mc E_\dc)$ (black in Fig.~\ref{fig:gt}). Even though the overall hybrid network $\mc G_N$ corresponds to a connected graph,  the ac and dc graphs $\mc G_\ac$ and $\mc G_\dc$ are not necessarily connected. Thus, we partition $\mc G_\ac$ into connected components $\mc G_\ac= \bigcup_{i=1}^{N_\ac} \mc G_\ac^i$  corresponding to $N_\ac$ subgrids, i.e., for all $i \in \N_{[1,N_\ac]}$, $\mc G_\ac^i=(\mc N_\ac^i \cup \mc N_{\ad}^i, \mc E_\ac^i)$ where $\mc E_\ac^i$ is the edge set and $\mc N^i_\ac$ and $\mc N_\ad^i$ denote the ac nodes and dc/ac nodes (i.e., the converter nodes from $\mc N_c$ that are part of the $i^{\text{th}}$ ac graph). Analogously, $\mc G_\dc= \bigcup_{i=1}^{N_\dc} \mc G_\dc^i$, where for $i \in \N_{[1,N_\dc]}$, $\mc G_\dc^i=(\mc N_\dc^i \cup \mc N_{\da}^i, \mc E_\dc^i)$. Finally, we note that dc/ac voltage source converters interfacing ac and dc subgrids are part of both their corresponding ac and dc graphs. 
	
\subsection{Network power flow}\label{subsec:networkmodel}
We use a linear model for the power flow and all variables denote deviations from their linearization point. To every ac node $l \in \mc N_\ac$ we associate a voltage phase angle deviation $\theta_l \in \mathbb{R}$ and a frequency deviation $\omega_l \in \mathbb{R}$; to every dc node $l \in \mc N_\dc$ we associate a dc voltage deviation $v_l \in \mathbb{R}$, and to every dc/ac node $l \in \mc N_\ad$ we associate a voltage phase angle deviation $\theta_l \in \mathbb{R}$ and a dc voltage deviation $v_l \in \mathbb{R}$. To every ac edge $l \in \mc E_\ac$ we assign an active power deviation $P_{\ac,l} \in \mathbb{R}$, and to every dc edge $l \in \mc E_\dc$ we assign a power deviation $P_{\dc,l} \in \mathbb{R}$. For small phase angle and dc voltage deviations as well as constant ac voltage magnitudes we obtain the linearized power flow model \cite[Sec. 6]{PWS-MAP:98}
\begin{align}\label{eq:network.model}
P_{\text{ac}} = L_{\text{ac}} \theta + P_{\text{d}_\ac}, \quad
P_{\text{dc}} = L_{\text{dc}} v+ P_{\text{d}_\dc},
\end{align} 
where $L_{\text{ac}}$ is the Laplacian matrix of the graph $\mc G_\ac$ with ac edge susceptances as edge weights, $L_{\text{dc}}$ is the Laplacian matrix of the graph $\mc G_\dc$ with dc edge conductances as edge weights, and the vectors $\theta \in \mathbb{R}^{|\mc N_\ac| + |\mc N_\text{c}|}$ and $v \in \mathbb{R}^{|\mc N_\dc| +|\mc N_\text{c}|}$ collect the ac voltage phase angles and dc voltages of the different nodes. Finally, $P_{\text{d}_\ac}\in \mathbb{R}^{|\mc N_\ac|+|\mc N_\text{c}|}$ and $P_{\text{d}_\dc}\in \mathbb{R}^{|\mc N_\dc|+|\mc N_\text{c}|}$ model variations in load at the ac and dc nodes.
	
\subsection{Device models}
In the following, we present reduced-order device models that will be used to obtain the overall power system model in Sec. \ref{subsec:power.system.model}. Before proceeding, we note that machines and converters have losses that, in theory, render the power system stable. However, in practice, they are often too small (e.g., flywheel friction losses or HVDC converter losses) to rely on them for stability. Therefore, we assume that coefficients $D_l$ and $G_l$ in the following equations that model device losses are only used to model significant damping (e.g., frequency depended loads) and zero for negligible parasitic losses.
\subsubsection{Synchronous machines}
For all $l \in \mc N_\ac$, we use the second order machine dynamics \cite[Section 5]{PWS-MAP:98}
\begin{subequations}\label{eq:machine.model}
\begin{align}
\ddt \theta_l &= \omega_l,\\
M_l \ddt \omega_l &= -D_l \omega_l + P_l - P_{\ac,l} \label{eq:machine.model.omega.}. 
\end{align}	  
\end{subequations}
Here, $M_l \in \R_{>0}$ and $D_l \in \R_{\geq 0}$ model the machine inertia and losses. Moreover, $P_{\ac,l}$ is the ac active power deviation, and $P_l \in \R$ is the deviation of the mechanical power applied to the machine rotor. 
If the machine is not interfacing generation (e.g., synchronous condenser or flywheel), then $P_l=0$. Otherwise, we use the turbine model 
\begin{align} \label{eq:turbine.model}
T_{\g,l} \ddt P_l = -P_l - k_{\g,l} \omega_l,
\end{align}
where $k_{\g,l} \in \R_{\geq 0}$ is the linearized sensitivity of the turbine with respect to changes in frequency (e.g., governor gain of a steam turbine) and $T_{\g,l} \in \R_{>0}$ is its time constant. Conceptually, \eqref{eq:turbine.model} can also be used to model a wind turbine. Fig.~\ref{fig:wteff} shows the power generated by a wind turbine with zero blade pitch angle as a function of the rotor speed $\omega$ and wind speed \cite{ABL+2012}. Linearizing at the MPP (circle) results in $k_{\g,l}=0$. Linearizing at a higher turbine speed (triangle) results in $k_{\g,l} \in \R_{>0}$, and $T_{\g,l} \in \R_{>0}$ is an aerodynamic time constant \cite{KB2013}. A more detailed investigation of the complex dynamics of wind turbines utilizing blade pitch control is seen as an interesting area for future work.
\begin{figure}[t!]
\centering
\includegraphics[width=0.7\columnwidth]{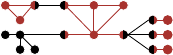} 
\caption{Example of a hybrid ac/dc network topology with ac nodes and lines (red), dc nodes and lines (black), and dc/ac nodes (red/black). \label{fig:gt}}
\end{figure}
\begin{figure}[t!]
\centering
\includegraphics[width=1\columnwidth]{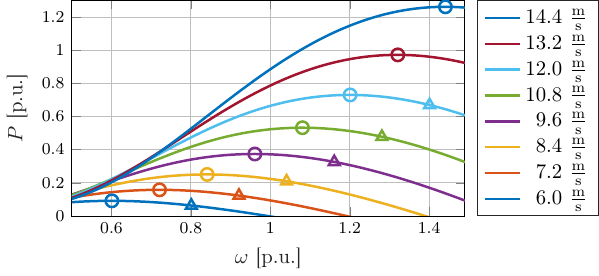}
\vspace{-1.5em}
\caption{Steady-state power generation $P$ of a wind turbine with zero blade pitch angle as a function of the rotor speed $\omega$ and wind speed, MPP (circle), and operating points with power reserves (triangle). \label{fig:wteff}}
\end{figure}
\subsubsection{DC nodes and dc sources}
For all $l \in \mc N_\dc$, we use the following dc bus dynamics
\begin{align} \label{eq:dc.model}
C_l \ddt v_l &= -G_l v_l + P_l - P_{\dc,l} , 
\end{align}
where $C_l= c_{l,\text{dc}} v^\star_l \in \R_{>0}$, $c_{l,\text{dc}} \in \R_{>0}$ is the dc capacitance, and $v^\star_l \in \R_{>0}$ is the nominal dc voltage. Moreover, $G_l = g_{l,\text{dc}} v^\star_l \in \R_{\geq 0}$ where $g_{l,\text{dc}} \in \R_{\geq 0}$ is the dc conductance, $P_{\dc,l}$ is the deviation of the dc network power injection. If the dc bus is not interfacing generation then $P_l=0$. Otherwise, we model $P_l$ by
\begin{align} \label{eq:power.source}
T_{\g,l} \ddt P_l = -P_l - k_{\g,l} v_l,
\end{align}	
where $T_{\g,l}\in\R_{>0}$ is the dc source time constant and $k_{\g,l}\in\R_{\geq 0}$ is its sensitivity with respect to the dc voltage. For example, linearizing the power generation of a PV module at the MPP results in $k_{\g,l}=0$, while linearizing at an operating point with power reserves results in $k_{\g,l}\in\R_{>0}$ (see Fig.~\ref{fig:pvcharacteristic}).
	
\subsubsection{DC/AC voltage source converters}
Each dc/ac voltage source converter (VSC) with index $l \!\in\! \mc N_\cc$ modulates a dc voltage $v_l$ into an ac voltage. The angle and magnitude of the ac voltage are control inputs. For frequency stability analysis of transmission systems, the ac voltage magnitude is typically assumed to be constant \cite[Sec. 6]{PWS-MAP:98}. The dc-link capacitor dynamics are modeled by \cite{CGD17}
\begin{align} \label{eq:converter.model}
C_l \ddt v_l &= -G_l v_l + P_l - P_{\ac,l} - P_{\dc,l} , 
\end{align}
where $C_l=c_{l,\text{dc}}v^\star_l  \in \R_{>0}$ and $G_l=g_{l,\text{dc}}v^\star_l  \in \R_{\geq 0}$ are the (scaled) dc capacitance and conductance, and $P_l$, $P_{\dc,l}$, $P_{\ac,l}$, denote the dc source power and dc and ac network power injections. Note that $P_{\ac,l}$ is a function of the angle deviation $\theta_l$ (cf. Sec.~\ref{subsec:networkmodel}).  For power converters that do not interface power generation $P_l=0$, e.g., static synchronous compensators (STATCOM) used for reactive power control or HVDC converters \cite{GO+21}. Otherwise, the power generation $P_l$ (e.g., PV modules) is modeled by \eqref{eq:power.source}.
\subsubsection{Modeling complex devices through model composition}
A wide range of complex devices and topologies can be modeled through composition of the models developed in this section. For example, Fig.~\ref{fig:b2bwttop} shows a wind turbine interfaced by a synchronous machine and back-to-back power converters. Moreover, an offshore wind farm can be modeled by connecting multiple wind turbines to an ac subgrid that is connected to an onshore ac subgrid through a dc network.
	
\begin{figure}[h!!]
\centering
\includegraphics[width=0.4\textwidth]{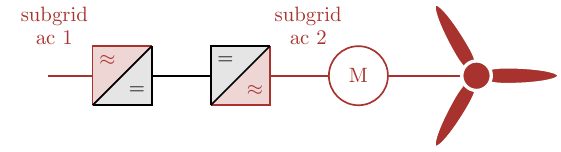}
\caption{Wind turbine interfaced by a synchronous machine and back-to-back power converters. The connection between the machine and the rotor side converter is modeled through an ac edge and the connection between the power converters is modeled through a dc edge. \label{fig:b2bwttop}}
\end{figure}

\section{Power-balancing dual-port GFM control}
The ac voltage angle dynamics of each dc/ac converter, are prescribed by the power-balancing dual-port GFM control \footnote{Typically, GFM control provides a voltage reference to cascaded inner controls. For system-level stability analysis the inner controls are commonly assumed to perfectly track the voltage reference and are neglected \cite{SG+20}.}
\begin{align} \label{eq:control.law}
\ddt \theta_l = -m_{p,l} P_{\ac,l} + k_{\theta,l} v_l,
\end{align}
with $P_\ac\!-\!f$ droop coefficient $m_{p,l} \!\in\! \R_{>0}$, $v_\dc\!-\!f$ droop coefficient $k_{\theta,l} \!\in\! \R_{>0}$, and active power and dc voltage deviations $P_{\ac,l}$ and $v_l$. 

Broadly speaking, the $P_\ac\!-\!f$ droop term contributes to angle/frequency synchronization, while the $v_\dc\!-\!f$ droop term stabilizes the dc voltage through the converter ac side. If the converter dc voltage is tightly controlled by a dc source (i.e., $v_l \approx 0$), \eqref{eq:control.law} resembles standard ac-GFM $P_\ac\!-\!f$ droop control \cite{CDA93}. In contrast, if the dc voltage is not controlled by a dc source (e.g., when interfacing renewables operating at the MPP) or there is no dc source (e.g., in a STATCOM), the $v_\dc\!-\!f$ droop term dominates the response and \eqref{eq:control.law} resembles ac-GFL control that adjusts the ac power to stabilize the dc voltage. More generally, between these two extreme cases, \eqref{eq:control.law} translates the signals indicating power imbalance in ac networks (i.e., ac frequency and ac power imbalance) and dc networks (i.e., dc voltage deviation) to each other. In other words, power imbalances in any ac or dc network propagate to the power sources in all ac and dc network (see Fig.~\ref{fig:imbalance}) resulting in a stabilizing response by power sources with $k_{\g,i}>0$. In contrast, using standard ac-GFM control (e.g., \cite{CDA93,DSF2015,SG+20}), the dc voltage is not controlled through the converter ac side and needs to be stabilized through the converter dc side.

Next, we define the \emph{effective} droop gains $\kappa_{p,l}\coloneqq m_{p,l}+k_{\theta,l}/(k_{\g,l} + G_l)$ and $\kappa_{v,l}\coloneqq m_{p,l}(k_{\g,l}+G_l)+k_{\theta,l}$. For converter-interfaced generation, the steady-state relationship between frequency, active power, and dc voltage is given by
\begin{align}
 \omega_l = -\kappa_{p,l} P_{\ac,l},\quad 
 \omega_l =  \kappa_{v,l} v_l.
\end{align}
System operators typically determine $\kappa_{p,l}$, and the converter design determines a lower bound on $\kappa_{v,l}$ that ensures a limited dc voltage deviation for a given maximum frequency deviation. In other words, like standard $P_\ac\!-\!f$ droop control, the gains are fully determined by steady-state specifications. For a converter with $k_\g=0$ (e.g., interfacing HVDC), $\kappa_{v,l}=k_{\theta,l}$ and $m_{p,l}$, can be tuned independently of device steady-state specifications. Finally, for a machine with $k_\g>0$, $\kappa_{p,l}=1/k_\g$.

\section{Overall power system model}\label{subsec:power.system.model}
The overall power system model combines the ac and dc transmission network model \eqref{eq:network.model}, synchronous machine model \eqref{eq:machine.model}, power converter model \eqref{eq:converter.model} with power-balancing GFM control \eqref{eq:control.law}, and the power source models \eqref{eq:turbine.model} and \eqref{eq:power.source}.
	
We define a vector $\theta  \in \R^{|\mc N_\ac \cup \mc N_\cc|}$ that collects the angles of all ac nodes and converter nodes, a vector $\omega \in \R^{|\mc N_\ac|}$ that collects the frequencies of ac nodes, and a vector $v \in \R^{|\mc N_\cc \cup \mc N_\dc|}$ that collects the dc voltages of all dc/ac converter nodes and dc nodes. Next, we define the set $\mc N_\g \subseteq \mc N$ of power sources (i.e., a turbine or a dc power source) that responds to frequency or dc voltage deviations (i.e., $k_{\g,i}>0$ if $i \in \mc N_\g$) and $P \in \R^{|\mc N_\g|}$ collects their power generation. Additionally, the matrix $\mc I_{\g,\ac} \in \{0,1\}^{|\mc N_\ac| \times |\mc N_\g|}$ models the interconnection of machines and stabilizing mechanical power sources and
\begin{table*}[b!!]
\small
\begin{align} \label{eq:full.dynamics.change.of.coordiantes}
\!\!\!T \ddt \!\!\begin{bmatrix}
\eta \\ \omega \\ v\\ P \\ \bar{P}
\end{bmatrix} \!\!=\!\! \begin{bmatrix}
-(\mc I_{\ad} B_\ac)^\T \!M_p \mc I_{\ad} B_\ac \mc W_\ac  &  (\mc I_{\ac} B_\ac)^\T  & (\mc I_{\ad} B_\ac)^\T K_\theta  \mc I_{\da} & \mathbbl{0}_{(|\mc N_\ac \cup \mc N_\cc|)\times |\mc N_\g|} & \mathbbl{0}_{(|\mc N_\ac \cup \mc N_\cc|)\times |\bar{\mc N}_\g|} \\
- \mc I_{\ac} B_\ac \mc W_\ac & -D & \mathbbl{0}_{|\mc N_\ac| \times |\mc N_\ac|} & \mc I_{\g,\ac} & \bar{\mc I}_{\g,\ac} \\
- \mc I_{\da}^\T \mc I_{\ad} B_\ac \mc W_\ac  & \mathbbl{0}_{|\mc  N_\cc \cup \mc N_\ac|\times |\mc N_\ac|}  & -(G+L_\dc) & \mc I_{\g,\dc} & \bar{\mc I}_{\g,\dc} \\
\mathbbl{0}_{|\mc N_\g| \times |\mc E_\ac|} & -K_\g \mc I_{\g,\ac}^\T & -K_\g \mc I_{\g,\dc}^\T  & -I_{|\mc N_\g|} & \mathbbl{0}_{|\mc N_\g| \times |\bar{\mc N}_\g|} \\
\mathbbl{0}_{|\bar{\mc N}_\g| \times |\mc E_\ac|} & \mathbbl{0}_{|\bar{\mc N}_\g| \times |\mc N_\ac|}  & \mathbbl{0}_{|\bar{\mc N}_\g| \times |\mc N_c|} & \mathbbl{0}_{|\bar{\mc N}_\g| \times |\mc N_\g|} & -I_{|\bar{\mc N}_\g|} 
\end{bmatrix} \!\!\! \begin{bmatrix}
\eta \\ \omega \\ v\\ P \\ \bar{P}
\end{bmatrix}\!\!
\end{align}
\end{table*}
\begin{align*}
\{\mc I_{\g,\ac}\}_{(i,j)}=\left \{ \begin{array}{cc} 1, & i \in \mc N_\ac, \ j \in  \mc N_\g, \\
0, & i \in \mc N_\ac, \ j \notin  \mc N_\g \end{array} \right \}, 
\end{align*}
i.e., a machine with index $i$ is connected to a turbine $j$ with $k_{\g,j}>0$ iff $\{\mc I_{\g,\ac}\}_{(i,j)}=1$. Similarly, $\mc I_{\g,\dc} \in \R^{|\mc N_\cc \cup \mc N_\dc| \times |\mc N_\g|}$ describes which dc/ac and dc nodes are connected to a stabilizing dc power source, i.e., 
\begin{align*}
\{\mc I_{\g,\dc}\}_{(i,j)}=\left \{ \begin{array}{cc} 1, & i \in \mc N_\cc \cup \mc N_\dc, \ j \in  \mc N_\g, \\
0, & i \in \mc N_\cc \cup \mc N_\dc, \ j \notin  \mc N_\g \end{array} \right \}.
\end{align*}
The set of power sources with dynamics \eqref{eq:power.source} and \eqref{eq:turbine.model} that do not respond to frequency or dc voltage deviations is denoted by $\bar{\mc  N}_\g \subseteq \mc N$ (i.e., $k_{\g,i}=0$ if $i \in \bar{\mc N}_\g$) and $\bar{P}\in \R^{|\bar{\mc N}_\g|}$ denotes their power generation. 
Analogously to $\mc I_{\g,\ac}$ and $\mc I_{\g,\dc}$,
the matrices $\bar{\mc I}_{\g,\ac}\in \{0,1\}^{|\mc N_\ac| \times | \bar{\mc N}_\g|}$ and $\bar{\mc I}_{\g,\dc}\in \{0,1\}^{|\mc N_c| \times | \bar{\mc N}_\g|}$ model the interconnection between machines, converters, power sources in $\bar{\mc  N}_\g$. For notational convenience, we define matrices $\mc I_\ac \in \{0,1\}^{|\mc N_\ac| \times |\mc N_\ac \cup \mc N_\cc|}$ and $\mc I_\ad \in \{0,1\}^{|\mc N_\cc| \times |\mc N_\cc \cup \mc N_\ac|}$ to extract machine and converter angles from the overall angle vector $\theta$, e.g., $\mc I_\ac \theta$ is the vector of all machine angles. Similarly, $\mc I_\da \!\in\! \{0,1\}^{|\mc N_\cc| \times |\mc N_\cc \cup \mc N_\dc|}$, and $\mc I_\dc \!\in\! \{0,1\}^{|\mc N_\dc| \times |\mc N_\cc \cup \mc N_\dc|}$ extract converter dc voltages and dc node voltages from the vector $v$. 

To facilitate the stability analysis, we change coordinates from absolute angles $\theta$ to angle differences i.e.,  $\eta \coloneqq B_\ac^\T \theta$ \cite[cf. Sec. III]{MDPS+17}, where $B_\ac \in \{ -1,0,1 \}^{|\mc N_\ac \cup \mc N_\ad| \times |\mc E_\ac|}$ is the incidence matrix of the ac graph $\mc G_\ac$. 

Finally, the overall model of a hybrid power system is given by \eqref{eq:full.dynamics.change.of.coordiantes} with $T\coloneqq \blkdiag \{ I_{|\mc N_\ac \cup \mc N_\cc|}, M, C, T_\g,\bar{T}_\g \}$ and machine inertia $M\coloneqq \diag \{ M_i \}_{i=1}^{|\mc N_\ac|} \succ 0$, dc capacitance $C=\diag \{ C_{i} \}_{i=1}^{|\mc N_\cc \cup \mc N_\dc|} \succ 0$, and power generation time constants $T_\g \coloneqq \diag\{T_{\g,i}\} \succ 0$  and $\bar{T}_\g \coloneqq \diag\{\bar{T}_{\g,i}\} \succ 0$. Moreover, $\mc W_\ac$ is a diagonal matrix of ac edge weights, and $L_\dc$ is the dc graph Laplacian. Finally, $M_p\coloneqq \diag \{m_{p,i}\}_{i=1}^{|\mc N_\cc|} \!\succ\! 0$ and $K_\theta \coloneqq \diag \{k_{\theta,i}\}_{i=1}^{|\mc N_\cc|}  \succ 0$ collect the converter control gains, $D \coloneqq \diag \{D_i\}_{i=1}^{|\mc N_{\ac}|} \succeq 0$ and $G= \diag\{ G_{i}\}_{i=1}^{|\mc N_\ad \cup \mc N_\dc|} \!\succeq \! 0$ collect machine and converter losses, and $K_\g=\diag\{k_{\g,i}\}_{i=1}^{|\mc N_\g|} \succ 0$ collects the power source sensitivities. 

\section{Stability analysis}\label{sec:stability}
Typically, conditions for frequency stability of multi-converter/multi-machine ac power systems exploit passivity, and do not consider the converter dc side. However, in our setting the dc source and network dynamics play a crucial role and individual devices may not be passive. The framework in \cite{PM2019} does not rely on passivity of the nodes/devices, but aims to establish asymptotic stability and robustness guarantees for multi-converter/multi-machine ac networks with arbitrary connected topologies. While this framework is very general, it is not readily applicable to our setting because we consider multiple disjoint ac and dc networks that are interconnected through power converters. Moreover, the following example demonstrates that, even when restricting the focus to a single multi-machine ac network (i.e., synchronous generators and synchronous condensers), asymptotic stability can, in general, not be guaranteed using only device parameters. 	
\begin{example}{\bf(Stability and network parameters)}  \label{example:three.machines}
We consider an ac network that consists of two machines without damping and one machine with damping. The dynamics and network topology are given by
\begin{align*}
\ddt \begin{bmatrix}
	\eta \\ M \omega
\end{bmatrix} =
\begin{bmatrix}  
\mathbbl{0}_{2 \times 2} & B_\ac^\T  \\
- B_\ac \mc W_\ac & -D
\end{bmatrix} \begin{bmatrix}
\eta \\ \omega
\end{bmatrix},\; B_\ac \coloneqq \begin{bmatrix} 1 & 1 \\ -1 &0 \\ 0 &-1 \end{bmatrix},
\end{align*}
with line susceptances $\mc W_\ac =\diag \{b_1,b_2\}$ and damping $D=\diag\{d_1,0,0\} \succeq 0$. For $b_2/b_1=m_3/m_2$, the solution of the dynamics starting from the initial condition $\eta(0)=c(\tfrac{b_2}{b_1},0)$ and $\omega(0)=\mathbbl{0}_3$ is given by $\zeta=\sqrt{b_1/m_2}$ and $\eta_1(t)=\tfrac{b_2}{b_1} \cos(\zeta t)$, $\eta_2(t)=-\cos(\zeta t)$, $\omega_1(t)=0$, $\omega_2(t)=\frac{b_2}{\sqrt{b_1  m_2}} \sin(\zeta t)$, and $\omega_3(t)=-\zeta\sin(\zeta t)$. Thus, the solution does not converge and the system is not asymptotically stable. 
\end{example}
In this example, for every choice of machine parameters (i.e., $m_1$, $m_2$, $m_3$, and $d_1$), there exists network parameters (i.e., $b_1$ and $b_2$) such that the multi-machine ac system is not asymptotically stable. In other words, in our setting, stability conditions that only require connectedness of the network can, in general, not be obtained.

Motivated by this observation, Sec.~\ref{subsec:scalable} develops stability conditions that can be verified for each ac subgrid and only require partial knowledge of the graph of each ac system, i.e., do not require knowledge of control gains and line susceptances. Sec.~\ref{subsec:mainresult} provides the stability result and Sec.~\ref{subsec:illustrative.examples} provides examples that illustrate the stability conditions.

\subsection{$N-\mu$ stability conditions}\label{subsec:scalable}

To begin with, we require the following condition that allows to guarantee the stability within each dc network.
\begin{condition} {\bf(Consistent $\boldsymbol{v}_\dc\!-\!\boldsymbol{f}$ droop)}\label{cond:identical.k.theta.gains} For all $i \!\in\! \N_{[1,N_\dc]} $ and all $(n,l) \!\in\! \mc N^i_{\da}\! \times \! \mc N^i_{\da}$ it holds that $k_{\theta,n}\!=\! k_{\theta,l} \coloneqq k_\theta^i$.
\end{condition}
This condition requires the per unit $v_\dc\!-\!f$  droop gains of devices connected to the same dc subgrid to be equal and ensures a consistent mapping of frequency deviations and dc voltage deviations (i.e., $\omega_l/v_l = k_\theta^i$ for all $l \in N^i_{\da}$) between ac and dc subgrids at the nominal power flow (i.e., if $P_{\ac,l}=0$). This condition is important to ensure frequency and dc voltage coherency, i.e., that individual converters do not deviate too much from the average frequency (dc voltage) of an ac subgrid (dc subgrid), and to avoid using control gains that induce excessive power flows due to incoherent frequencies (dc voltages) at different nodes of an ac subgrid (dc subgrid). While a detailed analysis (i.e., including measurement noise) is beyond the scope of this work, the following example illustrates the need for consistent $v_\dc\!-\!f$  droop gains.
\begin{example}{\bf(Circulating power flow)} Consider two lossless dc/ac converters (i.e., $G_l=0$) connected to the same ac and dc subgrid. Moreover, the system is in steady state and no other devices are connected to the dc subgrid. Thus, $P_{\dc,1}=g_\dc(v_1-v_2)$ and $P_{\ac,1}\!=\!-P_{\dc,1}\!=\!P_{\dc,2}\!=\!-P_{\ac,2}$. Using $\omega_s$ to denote the synchronous steady-state frequency,  \eqref{eq:control.law} becomes $v_l=k_{\theta,l}^{-1}(\omega_s+m_{p,l}P_{\ac,l})$ for $l\in\{1,2\}$ and we obtain $P_{\dc,1}=g_\dc(k_{\theta,1}^{-1}(\omega_s+m_{p,1}P_{\ac,1}) - k_{\theta,2}^{-1}(\omega_s+m_{p,2}P_{\ac,2}))$. Solving for $P_{\dc,1}$ results in $P_{\dc,1} = (1+\frac{g_\dc m_{p,1}}{k_{\theta,1}}-\frac{g_\dc m_{p,2}}{k_{\theta,2}})^{-1} (\frac{g_\dc}{k_{\theta,1}}-\frac{g_\dc}{k_{\theta,2}}) \omega_s$, i.e., $P_{\dc,1}=0$ for $\omega_s \neq 0$ if and only if $k_{\theta,1}=k_{\theta,2}$. In contrast, if $k_{\theta,1}\neq k_{\theta,2}$ and $\omega_s \neq 0$ power will circulate through from the ac subgrid to the dc subgrid, and back to the ac subgrid.
\end{example}
Given a security level $(\mu_{\mc N},\mu_{\mc E}) \in \mathbb{N}_0 \times \mathbb{N}_0$ (e.g., specified by a system operator), we develop conditions to verify stability for all systems obtained by deleting up to $\mu_{\mc N}$ nodes and $\mu_{\mc E}$ edges from \eqref{eq:full.dynamics.change.of.coordiantes}. For brevity of the presentation we make the following Assumption that rules out a \emph{system split} for the overall system and each ac subsystem when deleting $\mu_{\mc N}$ nodes and $\mu_{\mc E}$ edges.
\begin{assumption}{\bf{($\boldsymbol{N\!-\!\mu}$ connectivity)}}\label{assum:N-mu}
Consider the security level $(\mu_{\mc N},\mu_{\mc E})$. It holds that $\mc G_N$ and its connected ac components $\mc G^i_\ac$, $i \in \N_{[1,N_\ac]}$, are connected when deleting any $\mu_{\mc N}$ nodes (and their associated edges) and any $\mu_{\mc E}$ edges.
\end{assumption}
Guaranteeing stability after deleting up to $\mu_{\mc N}$ nodes requires the initial number of devices with stabilizing response (i.e., $k_{g,l}>0$, $G_l>0$, or $D_l>0$) to be larger than $\mu_{\mc N}$.
This requirement is formalized in the next Assumption. To this end, for all ac subgrids $i\in \N_{[1,N_\ac]}$, we define node sets that collect nodes $\mc N_{\ac^\el}^i$ and $\mc N_{\ad^\el}^i$ with significant losses (i.e., $D_l>0$ and $G_l>0$), nodes $\mc N_{\ac^\g}^i$ and $\mc N_{\ad^\g}^i$ that are connected to power sources with $k_{\g,l}>0$, and the remaining nodes $\mc N_{\ac^\oo}^i=\mc N_\ac^i\setminus (\mc N_{\ac^\el}^i \cup \mc N_{\ac^\g})$, $\mc N_{\ad^\oo}^i=\mc N_\ad^i\setminus (\mc N_{\ad^\el}^i \cup \mc N_{\ad^\g})$. Similarly, for each dc subgrid $i \in \N_{[1,N_\dc]}$, we define the node sets $\mc N_{\dc^\el}^i$, $\mc N_{\da^\el}^i$, $\mc N_{\dc^\g}^i$, $\mc N_{\da^\g}^i$, $\mc N_{\dc^\oo}^i$, and $\mc N_{\da^\oo}^i$. 
\begin{assumption}{\bf{(Frequency \& dc voltage stabilization)}}\label{assump:onestab} 
One of the following holds for $\mc G_N$:
\begin{enumerate}
\item $\exists i \in \N_{[1,N_\ac]}:| \mc N^i_{\ac^\g} \cup \mc N^i_{\ac^\el} \cup \mc N^i_{\ad^\el} \cup \mc N^i_{\ad^\g}| > \mu_{\mc N}$, \label{assump:onestab.ac}
\item $\exists i \in \N_{[1,N_\dc]}: |\mc N^i_{\da^\g} \cup \mc N^i_{\da^\el} \cup \mc N^i_{\dc^\el} \cup \mc N^i_{\dc^\g}| > \mu_{\mc N}$. \label{assump:onestab.dc}
\end{enumerate}
\end{assumption}
As illustrated in Example~\ref{example:three.machines}, in general one cannot expect to find conditions that ensure stability for all connected network topologies. To formalize the stability conditions on the network topology and clarify the roles of different devices, the nodes in each ac subgrid are partitioned into different groups depending on whether an ac subgrid is machine-dominated or converter-dominated.
\begin{definition}{\bf(Partitioning of $\boldsymbol{\mc N}^{\boldsymbol{i}}_\ac$)}\label{def:partition}
For every ac subgrid $i \in \N_{[1,N_\ac]}$ of $\mc G_N$ we partition $\mc N^i_\ac \cup \mc N^i_\ad$ as follows:
\begin{enumerate}
\item $|\mc N_\ad^i|+ \mu_{\mc N} < |\mc N_\ac^i|$: ${\mc D}^i\coloneqq\mc N_{\ac^\el}^i \cup \mc N_{\ac^\g}^i \cup \mc N_{\ad^\el}^i \cup \mc N_{\ad^\g}^i$, ${\mc C}^i\coloneqq \mc N_{\ac^\oo}^i$, and  $\mc F^i=\mc N^i_{\ad^\oo}$, \label{def:less.converters}
\item $|\mc N_\ad^i|\geq |\mc N_\ac^i|+\mu_{\mc N}$: ${\mc D}^i\coloneqq \mc N_{\ad}^i$, ${\mc C}^i\coloneqq \mc N_{\ac}^i$, and $\mc F^i=\emptyset$. \label{def:more.converters}
\end{enumerate}
\end{definition}
In a machine-dominated system $\mc D^i$ contains nodes that contribute to frequency stabilization, $\mc C^i$ contains machines without damping (e.g., synchronous condensers), and $\mc F^i$ collects converters without damping (e.g., HVDC converters). In contrast, in a converter-dominated system $\mc D^i$ contains all converters, and $\mc C^i$ all machines. Broadly speaking, this partitioning reflects the dominant stabilization mechanism in each case (i.e., frequency regulation and angle synchronization through $P-f$ droop). Next, we leverage the properties of the nodes $\mc D^i$ to establish synchronization of the nodes $\mc C^i$ and stability the overall system. To this end, we first define the subgraph $\bar{\mc G}^i_0$.
\begin{definition}{\bf(Reduced ac subgrid graph)}\label{def:graph}
For all $i \in \N_{[1,N_\ac]}$, we define the graph $\bar{\mc G}^i_0$ with node set $\bar{\mc N}_0  \coloneqq \mc N_\ac^i \cup \mc N_\ad^i$ and edge set $\bar{\mc E}_0\coloneqq \mc E_\ac^i \setminus (({{\mc C}}^i \times {\mc C}^i)\cup ({\mc D}^i\times {\mc D}^i))$.  
\end{definition}
Notably, for a converter-dominated subgrid, the graph $\bar{\mc G}^i_0$ only contains connections between converters and machines. On the other hand, for a machine-dominated subgrid, the graph $\bar{\mc G}^i_0$ only contains connections between devices that contribute to stabilizing frequency and devices that do not. 

Finally, Algorithm~\ref{alg:removal} identifies ac subgrid topologies for which stability can be guaranteed independently of the exact parameters of the connections (i.e., line susceptances) and control gains. To this end, a node $l \in \mc N$ is defined to be a single-edge node iff there exists only one $j \in \mc N$ s.t. $(l,j)\in \mc E $. Moreover, we define $\mu_{\max} \coloneqq \max\{\mu_{\mc N},\mu_{\mc E}\}$.
\begin{algorithm}[!htpb]
	\caption{Node removal for $i \in \N_{[1,N_\ac]}$}
	\begin{algorithmic}[1]
		\State \textbf{Set} $\bar{\mc G}^i_0$ as in Definition~\ref{def:graph}, $\bar{{\mc C}}^i_0 \coloneqq {\mc C}^i$, and $k\coloneqq 0$
		\While{there exists $\mu_{\max}+1$ single-edge nodes $l \in {\mc D}^i$ with an edge to a node $j \in \bar{{\mc C}}^i_k$}
		\State $\bar{{\mc C}}_{k+1}^i \coloneqq \bar{{\mc C}}_k^i \setminus \{j\}$
		\State $\bar{\mc E}_{k+1}\coloneqq \bar{\mc E}_k \setminus  \{(l,j)\}$, where $(l,j) \in \bar{\mc E}_k \cap ({\mc D}^i \times \bar{{\mc C}}_k^i)$
		\State $\bar{\mc N }_{k+1}^i \coloneqq \bar{\mc N}_k^i \setminus \{j\}$, $\bar{\mc G}^i_{k+1}=(\bar{\mc N}_{k+1}^i,\bar{\mc E}_{k+1}^i)$
		\State $k\coloneqq k+1$
		\EndWhile
	\end{algorithmic}
	\label{alg:removal}
\end{algorithm}
Algorithm~\ref{alg:removal} iteratively removes nodes from $\bar{\mc C}^i_{k} \subseteq \bar{\mc C}^i_0$ for which frequency synchronization to a node in $\mc D^i$ can be guaranteed, e.g., in a  converter-dominated subgrid, machines that synchronize with ac/dc converters are deleted. In a machine-dominated subgrid, machines that do not stabilize frequency but synchronize to converters or machines that stabilize frequency are deleted. Broadly speaking, the algorithm terminates with $\bar{{\mc C}}^i_K =\emptyset$, if devices that do not contribute to stabilizing frequency are sufficiently well connected to devices that stabilize frequency.

\subsection{Asymptotic stability of hybrid ac/dc power systems}\label{subsec:mainresult}
We first note that \eqref{eq:full.dynamics.change.of.coordiantes} is a cascaded system and the dynamics of $\bar{P}$ are trivially asymptotically stable. Thus, to show stability of the overall system, we first establish stability of the dynamics \eqref{eq:full.dynamics.change.of.coordiantes} when $\bar{P}\overset{!}{=}\mathbbl{0}_{\bar{\mc N}_\g}$. To this end, we define the LaSalle function $V \coloneqq V_\eta+V_\omega+V_v+V_P$, with $V_\eta \coloneqq \frac{1}{2}{\eta} ^\T \mc W_\ac  \eta $, $V_\omega \coloneqq   \frac{1}{2} {\omega}^\T M \omega$, $V_v \coloneqq   \frac{1}{2} {v} ^\T \tilde{K}_\theta  C v $, and 	$V_P \coloneqq \frac{1}{2} {P}^\T ({\mc I_{\g,\ac}}^\T {\mc I_{\g,\ac}}+ {\mc I_{\g,\dc}}^\T \tilde{K}_\theta {\mc I_{\g,\dc}}) {K_\g}^{-1} T_\g P$, where $\tilde{K}_\theta\in \R^{|\mc N_\cc \cup \mc N_\dc| \times |\mc N_\cc \cup \mc N_\dc |}$ is defined as $\{\tilde{K}_\theta\}_{(l,l)}\coloneqq k_\theta^i$, for all $i \in \N_{[1,N_\dc]}$, and all $l \in \mc N_\dc^i \cup \mc N_\da^i$ and zero otherwise. Next, we show that $V$ is positive definite with negative semi-definite derivative, along trajectories \eqref{eq:full.dynamics.change.of.coordiantes} when $\bar{P}=\mathbbl{0_{\bar{\mc N}_\g}}$.
\begin{proposition} \label{prop:derivative.expression} {\bf(LaSalle function)} 
Under Condition~\ref{cond:identical.k.theta.gains}, the function $V$ is positive definite and its time derivative of $V$ along the trajectories of \eqref{eq:full.dynamics.change.of.coordiantes} restricted to $\bar{P}=\mathbbl{0}_{|\bar{\mc N}_\g|}$ satisfies 
\begin{align} \label{eq:LaSalles.fcn.derivative}
\begin{split} 
\ddt{V} =& -\mc \eta^\T \mc W_\ac B^\T_\ac \mc I^\T _{\ad} M_p\mc  I_{\ad} B_\ac \mc W_\ac  \eta  -  {\omega}^\T D \omega\! \\
&-\frac{1}{2} {v} ^\T \left (\tilde{K_\theta} (G+L_\dc)+ (G\!+\!L_\dc) \tilde{K}_\theta  \right) v\!\\
&- P^\T (\mc I_{\g,\ac}^\T \mc I_{\g,\ac} + \mc I_{\g,\dc}^\T  \tilde{K}_\theta\mc I_{\g,\dc})K_\g^{-1} P \leq 0.
\end{split}
\end{align}
\end{proposition}
The proof is given in the Appendix. From \eqref{eq:LaSalles.fcn.derivative} one can conclude that, when $\bar{P}=\mathbbl{0}_{|\bar{\mc N}_\g|}$, all dc voltages within every dc subgrid synchronize, all power deviations $P$ converge to zero, the frequency deviation of the machines with losses and dc voltage deviation of nodes with losses converge to zero, and the angles of the machines and dc/ac converters partially synchronize. To establish convergence of the remaining variables the following proposition characterizes the largest invariant $\mc M$ set contained in $\mc S \coloneqq \{x \in  \R^n \vert \ddt V(x(t)) =0 \}$, i.e., $\mc M \subseteq \mc S$ and $x(t) \in \mc M$ for all $t\geq0$ if $x(0) \in \mc M$, where $x\coloneqq (\eta,\omega,v,P)\in \R^{n}$ and $n\coloneqq |\mc E_\ac|+ |\mc N_\ac| + |\mc N_\cc|+|\mc N_\dc|+|\mc N_g|$. 

\begin{proposition} {\bf(Largest invariant set)} \label{prop:max.inv.set}
If for all $i \in \N_{[1,N_\ac]}$, there exists $K^i \in \N$ such that Algorithm~\ref{alg:removal} terminates with $\bar{{\mc C}}_{K^i}^i \! \! = \! \emptyset$, Condition~\ref{cond:identical.k.theta.gains} holds, and Assumption~\ref{assum:N-mu} and \ref{assump:onestab} hold, then for all systems obtained by deleting at most $\mu_{\mc N}$ nodes and $\mu_{\mc E}$ edges from \eqref{eq:full.dynamics.change.of.coordiantes}, 
the origin is the largest invariant set contained in $\mc M$.
\end{proposition}
The proof is given in the Appendix. We can now state the following stability result.
\begin{theorem}{\bf{(Stability of hybrid ac/dc power systems)}} \label{thm:main.theorem}If for all $i \in \N_{[1,N_\ac]}$, there exists $K^i \in \N$ such that Algorithm~\ref{alg:removal} terminates with $\bar{{\mc C}}_{K^i}^i = \emptyset$ and Condition~\ref{cond:identical.k.theta.gains} and Assumptions~\ref{assum:N-mu} and \ref{assump:onestab} hold, then all systems obtained by deleting at most $\mu_{\mc N}$ nodes and $\mu_{\mc E}$ edges from \eqref{eq:full.dynamics.change.of.coordiantes} are asymptotically stable with respect to the origin.
\end{theorem}
\begin{proof} We first note that $V$ is positive definite, i.e., all sublevel sets of $V$ are bounded. Next, by Proposition~\ref{prop:derivative.expression} the derivative of $V$ along the trajectories of \eqref{eq:full.dynamics.change.of.coordiantes} is negative semidefinite if $\bar{P}\overset{!}{=}\mathbbl{0}_{|\bar{\mc N}_\g|}$ and Condition~\ref{cond:identical.k.theta.gains} holds. Thus, according to LaSalle invariance principle, \cite[Theorem 1]{L1960}, the trajectories \eqref{eq:full.dynamics.change.of.coordiantes} converge to the maximal invariant set contained in $\ddt V=0$ if $\bar{P}=\mathbbl{0}_{\bar{\mc N}_\g}$. According to Proposition~\ref{prop:max.inv.set}, for all systems obtained by deleting at most  $\mu_{\mc N}$ nodes and $\mu_{\mc E}$ edges from \eqref{eq:full.dynamics.change.of.coordiantes} the origin is  the maximal invariant set inside $\ddt V=0$. Consequently, the origin is a (uniformly) asymptotically stable equilibrium point of  all systems obtained by deleting at most most  $\mu_{\mc N}$ nodes and $\mu_{\mc E}$ edges from \eqref{eq:full.dynamics.change.of.coordiantes} when $\bar{P}= \mathbbl{0}_{|\bar{\mc N}_\g|}$. Next, we note that $\bar{T}_\g \ddt \bar{P}=-\bar{P}$ is (uniformly)  asymptotically stable and, applying \cite[Theorem 3.1]{V1980}, the proof follows.
\end{proof}
The proof of Theorem~\ref{thm:main.theorem} first establishes synchronization within each ac subgrid and dc subgrid. Next, Assumption~\ref{assump:onestab} is used to establish that at least one frequency deviation or dc voltage deviation converges to zero (for details see Proof of Proposition~\ref{prop:max.inv.set}). Together with frequency and dc voltage synchronization, this establishes stability of the overall hybrid ac/dc power system \eqref{eq:full.dynamics.change.of.coordiantes}. Notably, because Algorithm~\ref{alg:removal} does not use the control gains, Theorem~\ref{thm:main.theorem} ensures stability for all positive control gains that satisfy Condition~\ref{cond:identical.k.theta.gains}. Moreover, our conditions do not restrict the network parameters of ac subgrids that only contain converters (i.e., $\mc N^i_\ac = \emptyset$) because the dual-port GFM control phase angles within such an ac subgrid synchronize independently of the network parameters.

Applying Algorithm \ref{alg:removal} to simple graph structures (e.g., cycle graphs) and stability of the system with nominal graph ($\mu_{\max}\!=\!0$) results in the following corollary that highlights the importance of edges between nodes in ${\mc C}^i$ and ${\mc D}^i$.
\begin{corollary}{\bf(Simple network structures)} \label{cor:cycles}
For $\mu_{\max}=0$ and all $i \in \N_{[1,N_\ac]}$, if every node in ${\mc C}^i$ either
\begin{enumerate}
\item has an edge to a single-edge node in $\mc D^i$, or \label{cor:cycles.1}
\item is part of a cycle that contains a node from ${\mc C}^i$ that has an edge to a single edge-node in ${\mc D}^i$, \label{cor:cycles.2}
\end{enumerate}
then there exists $K^i \in \N$ such that Algorithm~\ref{alg:removal} terminates with $\bar{{\mc C}}_{K^i}^i = \emptyset$. 
\end{corollary}

In the most general case (e.g., allowing for synchronous condensers), the stability conditions depend on the system topology. By posing stronger requirements on the devices contained in each ac subgrid, the following topology independent result can be obtained.
	
\begin{corollary}{\bf(Topology independent conditions)} \label{cor:topology.independent}
If Condition \ref{cond:identical.k.theta.gains} and Assumption \ref{assum:N-mu} and \ref{assump:onestab} hold, and for all $i\in \N_{[1,N_\ac]}$ either $\mc N_\ac^i=\emptyset$ or $\mc N_{\ac^\oo}^i=\emptyset$ holds, then, all systems obtained by deleting at most $\mu_{\mc N}$ nodes and $\mu_{\mc E}$ edges from \eqref{eq:full.dynamics.change.of.coordiantes} are asymptotically stable with respect to the origin.
\end{corollary}
The conditions of Corollary \ref{cor:topology.independent} imply that $\mc C^i = \emptyset$ for all $i\in \N_{[1,N_\ac]}$ and the proof immediately follows from the fact that Algorithm \ref{alg:removal} terminates at the first iteration. Specifically, Corollary \ref{cor:topology.independent} requires that each ac network $i\in \N_{[1,N_\ac]}$ either (i) only contains converters (i.e., $\mc N_\ac^i=\emptyset$), or (ii) all machines are equipped with a turbine governor system or have significant losses (i.e., $\mc N_{\ac^\oo}^i=\emptyset$). 

We emphasize that Corollary \ref{cor:topology.independent} for $N_\ac=1$, $N_\dc=0$, and $\mc N_\ac^1=\emptyset$, recovers standard conditions for stability of a network of ac-GFM converters with the dc terminal modeled as constant voltage source. However, in addition, Corollary \ref{cor:topology.independent} also includes machines, converters interfacing renewable generation with limited flexibility, and dc transmission. Moreover, using dual-port GFM control, topology independent stability conditions for hybrid ac/dc power systems can be established. In contrast, standard ac-GFM and ac-GFL controls require assigning ac-GFM and ac-GFL controls to each converter interfacing ac and dc networks, which typically requires knowledge of the system topology \cite{GO+21}.

The next section illustrates that Corollary~\ref{cor:cycles} is directly applicable for several common scenarios.

\subsection{Illustrative examples} \label{subsec:illustrative.examples}
To illustrate the stability conditions, we consider the converter-dominated and machine-dominates subgrids shown in Fig.~\ref{fig:ac.area.types} and the back-to-back wind turbine shown in Fig.~\ref{fig:b2bwttop}.
\begin{figure}[h]
\centering
\subfloat[][]{\includegraphics[trim=0em 0.3em 0em 0.3em, clip,height=0.15\textwidth]{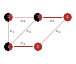} 
	\label{fig:a.converter.domimnated.example}}
\subfloat[][]{\includegraphics[trim=0em 0.3em 0em 0.3em, clip,height=0.15\textwidth]{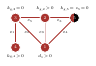} 
	\label{fig:b.machine.dominated.example}}
\caption{Examples for a converter-dominated ac subgrid (a) and machine-dominated ac subgrid (b). Edges not contained in $\bar{\mc G}^i_0$ are shown in light red.} \label{fig:ac.area.types}
\end{figure}
\subsubsection{Converter-dominated ac subgrid} applying Definition~\ref{def:partition}.\ref{def:more.converters} and Definition~\ref{def:graph}, ${\mc C}=\{4,5\}$, ${\mc D}=\{1,2,3\}$, and $\bar{\mc E}_0=\{e_4,e_5\}$ holds. Hence, in $\bar{\mc G}^i_0$, the nodes $\{1,3\} \in \mc C$ are single-edge nodes and from Corollary~\ref{cor:cycles}.\ref{cor:cycles.1} it directly follows that there exists $K^i\in \N$, such that Algorithm \ref{alg:removal} terminates with $\bar{\mc C}^i_{K^i} \!\! =\! \emptyset$.
\subsubsection{Machine-dominated ac subgrid} applying Definition~\ref{def:partition}.\ref{def:less.converters} we obtain ${\mc C}=\{3\}$, ${\mc D}=\{1,2,4\}$, $\mc F=\{5\}$, where, e.g., node $3$ models a synchronous condenser and node $5$ models a STATCOM, HVDC converter, or PV operating at its MPP. Using Definition~\ref{def:graph} it follows that $\bar{\mc E}_0=\{e_1,e_2,e_4,e_5,e_6\}$ and $\bar{\mc G}_0$ contains a cycle in which a node from ${\mc C}$ has an edge to a single-edge node from ${\mc D}$. From Corollary~\ref{cor:cycles}.\ref{cor:cycles.2}, there exists $K^i\!\in\! \N$, such that Algorithm \ref{alg:removal} terminates with $\bar{\mc C}^i_{K^i} \!\! =\! \emptyset$.
\subsubsection{Wind turbine or flywheel energy storage with back to back converter} Irrespective of the operating point of the wind turbine (i.e., $k_\g=0$ or $k_\g>0$), the subgrid ac~2 in Fig.~\ref{fig:b2bwttop} is converter-dominated and Corollary~\ref{cor:cycles}.\ref{cor:cycles.2} trivially applies. This result also applies to common flywheel energy storage systems (i.e., Fig.~\ref{fig:b2bwttop} without the wind turbine).  In addition, for ac 1 we require existence of $K^1 \! \in \!\N$ such that Algorithm \ref{alg:removal} terminates with $\bar{\mc C}^1_{K^1}=\emptyset$, and that the converter dc voltage - frequency droop gains satisfy Condition~\ref{cond:identical.k.theta.gains}. 
\subsubsection{Offshore wind farm} in an offshore wind farm containing wind turbines with back to back converters, the subgrid ac~1 in Fig.~\ref{fig:b2bwttop} only contains the grid-side converters of the wind turbines and an HVDC converter. In other words, $\mc N^1_\ac=\emptyset$,  and we only require condition~\ref{cond:identical.k.theta.gains} for the dc networks and for ac~1 we require existence of $K^1 \! \in \!\N$ such that Algorithm \ref{alg:removal} terminates with $\bar{\mc C}^1_{K^1}=\emptyset$.

\section{Case study} \label{sec:test.case.study}
In this section, we present a case study that combines ac and dc  transmission as well as conventional generation and PV. We emphasize that a mix of at least three different conventional GFM and GFL controls would be needed to operate this system. Subsequently, we apply our stability conditions and use a high-fidelity simulation to illustrate applying the proposed dual-port GFM control for all converters.

\subsection{Hybrid AC/DC system with renewable generation}
Consider the power system shown in Fig.~\ref{fig:ieee9bus.test.case} that consists of two IEEE-9 bus systems (ac~1 and ac~2) interconnected by an HVDC link (dc~1). The first ac subgrid contains conventional thermal generation (TG) interfaced by a synchronous machine (SM) with automatic voltage regulator (AVR) and power system stabilizer (PSS), a two-level voltage source converter (VSC) that interfaces a controllable dc source (e.g., a large-scale battery), a VSC that interfaces photovoltaics (PV), and a VSC that interfaces the subgrids ac~1 and dc~1. The HVDC link has a length of $310~\mathrm{km}$ and connects the subgrid ac~1 that has significant frequency control reserves (i.e., thermal generation and controllable dc source) with the subgrid ac~2 that only contains renewable generation (i.e., PV) interfaced by VSCs and a synchronous condenser (SC). The base frequency is  $f_b=50~\mathrm{Hz}$ for ac~1 and ac~2 and the base voltages and device base powers can be found in Table~\ref{tab:voltagelevels}. The parameters of the synchronous machines, low-voltage da/ac converters (i.e., VSC$_1$ to VSC$_4$), and low-voltage/high-voltage and medium-voltage/high-voltage transformer can be found in \cite[Table I]{TGA+20}. Moreover, the high-voltage/high-voltage transfomer admittance is $0.0027+j0.8~\mathrm{p.u.}$, and the HVDC converter and line parameters can be found in \cite{BDB+16} and \cite{JTM2015}. The PV systems $\text{PV}_1$, $\text{PV}_2$, and $\text{PV}_3$ aggregate 1200, 5000, and 3000 parallel strings of 90 or 100 ($\text{PV}_3$) modules (AUO PM060MBR). $\text{PV}_2$ and $\text{PV}_3$ operate above the MPP voltage to provide primary control (i.e., $k_\g>0$) while $\text{PV}_1$ operates at the MPP (i.e., $k_{\g}=0$) and resembles ac-GFL operation (see Fig.~\ref{fig:pvcharacteristic}).
\begin{figure}[htb!!]
 	\centering 
 	\includegraphics[width=1\columnwidth]{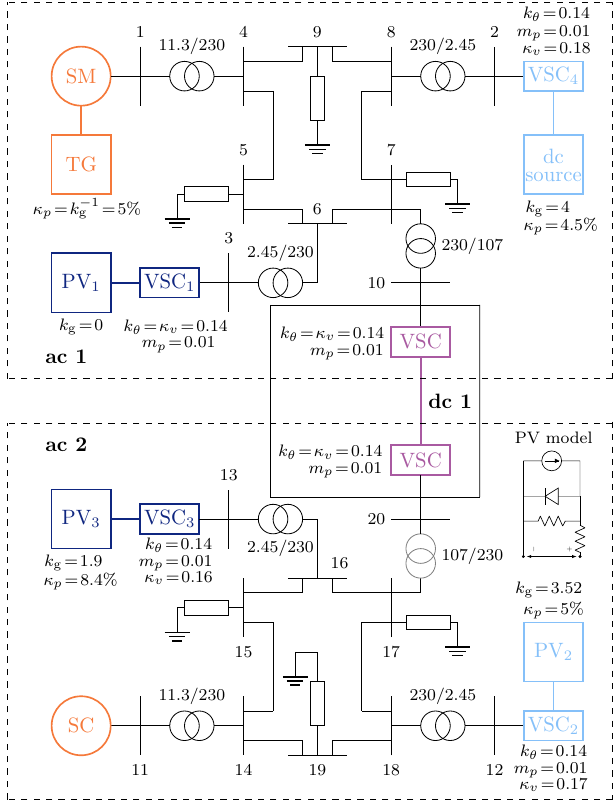} 
 	\caption{Power system with ac transmission, dc transmission, power converters, machines, photovoltaics, and conventional generation. The control gains and power source sensitivities are shown in per unit, and the base voltages for each transformer side are shown in Kilovolt.}\label{fig:ieee9bus.test.case}
\end{figure}
\begin{figure}[h!!]
	\centering 
	\includegraphics[trim=0em 0.3em 0em 0.3em,width=0.8\columnwidth]{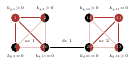} 
	\caption{Graph of the system in Fig.~\ref{fig:ieee9bus.test.case} after applying Kron reduction to the ac subgrids. Edges not contained in $\bar{\mc G}^i_0$ are shown in light red.\label{fig:ieee9graph}}
\end{figure}
\begin{table}[htb!]
	\caption{Base powers and base voltages\label{tab:voltagelevels}}
	\centering
	\resizebox{0.9\columnwidth}{!}{
		\begin{tabular}{ l c c c c}
			\toprule 
			& LV  & MV & HV & HVDC \\ \cmidrule{2-5}
			$V_b^\text{dc}$ [kV] & 2.4495 & $/$ & $/$ & 320 \\
			$V_b^\text{ac}$ [kV] & 0.8165& 11.27& 230&106.67 \\
			$S_b$ [MW] &  100 &100 &100 & 210 \\  \cmidrule{2-5}
			Bus no. & 2, 3, 12, 13& 1, 11& 4- 9, 14-19 & 10, 20 \\
			\bottomrule
		\end{tabular}}
\end{table}
\begin{figure*}[t!!!]
\centering
\includegraphics[width=1\textwidth]{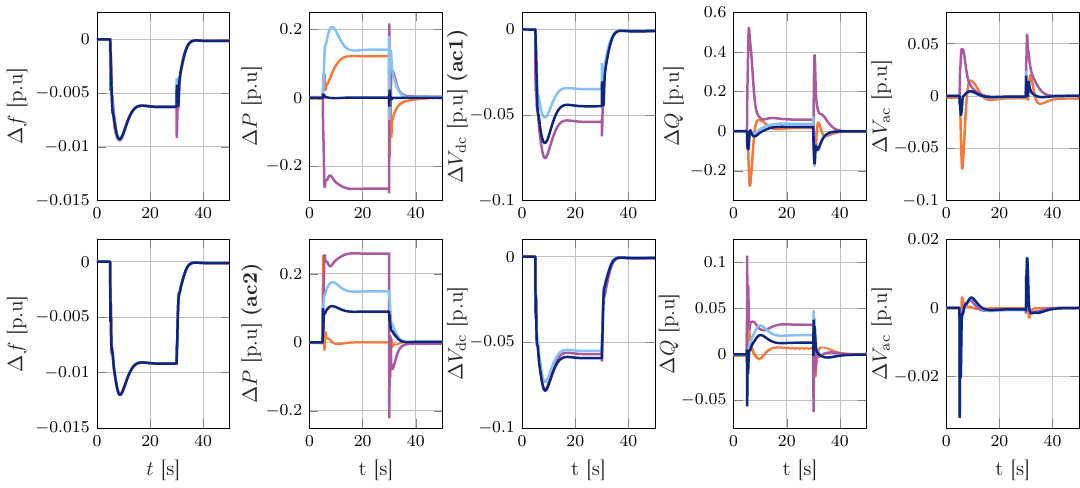} 
\caption{Simulation results for a $0.5~\mathrm{p.u.}$ load step at bus 17 and subsequent redispatch of the power system in Fig.~\ref{fig:ieee9bus.test.case}. The plots show the deviation from the setpoints for each device (the color-scheme is identical to Fig.~\ref{fig:ieee9bus.test.case}). \label{fig:whole.example}}
\end{figure*}
\subsection{Illustration of the  theoretical results}
To verify our assumptions we first apply Kron reduction \cite{DB2013} to the ac subgrids. The resulting graph of the overall system in Fig.~\ref{fig:ieee9bus.test.case} is shown in Fig.~\ref{fig:ieee9graph}. At the operating point under consideration, both ac~1 and ac~2 contain devices with $k_\g>0$ and Assumption~\ref{assump:onestab}.\ref{assump:onestab.ac} is trivially satisfied. Moreover, both ac subgrids are converter-dominated and, considering Definition~\ref{def:partition}.\ref{def:more.converters} and Definition~\ref{def:graph}, we obtain ${\mc D}^1=\{2,3,10\}$, ${\mc C}^1=\{1\}$, and $\bar{\mc E}_0^1=\{e_{1,3},e_{1,2},e_{1,10}\}$, as well as ${\mc D}^2=\{12,13,20\}$, ${\mc C}^2=\{11\}$, and $\bar{\mc E}_0^2=\{e_{11,13},e_{11,12},e_{11,20}\}$. Thus, for each $i \in \{1,2\}$, every node in $\mc C^i$ has an edge to a single-edge node in $\mc D^i$ and it follows from Corollary~\ref{cor:cycles} that there exist $K^i \!\in \!\N$ such that Algorithm \ref{alg:removal} terminates with $\bar{\mc C}^i_{K^i}\!\!=\!\emptyset$. Therefore, if the control gains satisfy Condition~\ref{cond:identical.k.theta.gains}, asymptotic stability of \eqref{eq:full.dynamics.change.of.coordiantes} is guaranteed by Theorem~\ref{thm:main.theorem}.
\subsection{Simulation results}
Finally, we use high-fidelity simulation results obtained using SimPowerSystems in MATLAB/Simulink to illustrate and validate the results. AC lines and the DC cable are modeled using the standard $\pi$-line dynamics \cite{PWS-MAP:98,JTM2015} and transformers are explicitly modeled using dynamical models. The simulation uses an $8^\text{th}$ order synchronous machine model with AC1A exciter model and automatic voltage regulator. In addition, the machine in ac~1 features a delta-omega power system stabilizer, and first order turbine model with 5\% speed droop. An averaged model of a two-level voltage source converter with RLC filter and cascaded inner current and voltage PI controllers are used (cf. \cite{TGA+20}). The voltage angle reference is provided by the dual-port GFM control \eqref{eq:control.law}. For simplicity, we use identical gains for all VSCs that result in an effective droop gain $\kappa_p$ of approximately $4.5\%$ to $8.5\%$ for converter-interfaced generation and $\kappa_v \geq 0.1$ for all VSCs (see Fig.~\ref{fig:ieee9bus.test.case}). Moreover, the voltage magnitude reference is obtained using standard reactive power-voltage droop with 3\% droop \cite{DSF2015}. Finally, the aggregated PV modules are modeled using an equivalent (see Fig.~\ref{fig:ieee9bus.test.case} and \cite{GRR2009}).

We simulate a load-step of $0.5~\mathrm{p.u.}$ at bus 17 and $t=5~\mathrm{s}$. At $t=30~\mathrm{s}$ the power setpoints of the turbine, dc source, and converters, are updated to return the system to the nominal frequency. The setpoints are provided in Table~\ref{tab:set.points} and the corresponding PV operating points are shown in Fig.~\ref{fig:pvcharacteristic}.

\begin{table}[b!!]
\caption{Setpoints (rounded)\label{tab:set.points}}
\resizebox{1\columnwidth}{!}{
\begin{tabular}{c l c c c c c c c c}
\toprule 
& & \multicolumn{4}{c}{Initial setpoints [p.u.]} & \multicolumn{4}{c}{Updated setpoints [p.u.]} \\  
& & $P$ & $V_\text{dc}$& $Q$ & $V_\text{ac}$& $P$ & $V_\text{dc}$& $Q$ & $V_\text{ac}$ \\ \midrule
\parbox[t]{2mm}{\multirow{4}{*}{\rotatebox[origin=c]{90}{ac 1}}} & SM & 1.07 &$/$ & $/$ & 1 & 1.28&$/$ & $/$&1 \\ 
& HVDC & -0.57 & 0.98 & 0.04& 1  &    -1.08 & 1& 0.16 & 1\\ 
&$\text{VSC}_1$ & 0.27 & 1.20 & -0.05 &1&  0.27 & 1.20 & -0.004 & 1\\  
&$\text{VSC}_4$ & 1.59 &0.98 & 0.02 & 1 & 1.9 & 0.97 & 0.11& 1 \\ \midrule
\parbox[t]{2mm}{\multirow{4}{*}{\rotatebox[origin=c]{90}{ac 2}}} & SC & 0 &$/$ & $/$ & 1 & 0&$/$ & $/$&1 \\ 
& HVDC & 0.56 & 0.98 & -0.08 & 1& 1.07& 0.99 & -0.05&1\\
&$\text{VSC}_2$ & 1.1& 1.32& -0.03& 1 & 1.1 & 1.32 &-0.03  & 1\\  
&$\text{VSC}_3$ & 0.7 & 1.46 & -0.09 &1 &  0.7 & 1.46 & -0.06 & 1  \\
\bottomrule
\end{tabular}
}
\end{table}

\begin{figure}[b!!!]
\centering 
\includegraphics[width=1\columnwidth]{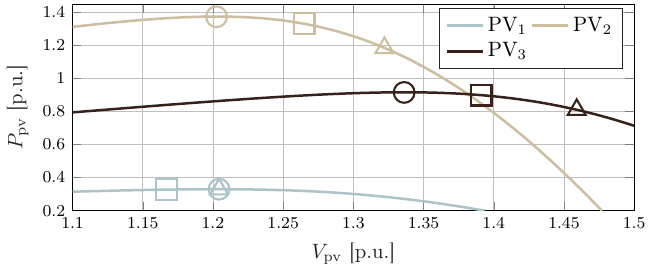} 
\caption{$P-V$ curve of $\text{PV}_1$, $\text{PV}_2$, and $\text{PV}_3$: MPP (circle), nominal operating point (triangle), and post-disturbance steady-state (square).\label{fig:pvcharacteristic}}  	
\end{figure}

The resulting deviations of the frequency, dc voltage, active power, ac voltage, and the reactive power from their setpoints is shown in Fig.~\ref{fig:whole.example}. We emphasize that the $0.5$~p.u. load step is very large and pushes the system to the boundary of the normal operating range. Nonetheless, the system dynamics are well-behaved. As predicted, power imbalances propagate to all ac and dc subgrids and the power sources share the additional load according to their sensitivities and converter control gains. 

The synchronous machine in ac~1 and VSC$_4$ provide primary frequency control (i.e., VSC$_4$ exhibits ac-GFM functions). Moreover, the PV systems $\text{PV}_2$ and $\text{PV}_3$ in the subgrid ac~2 increase their power generation and operate closer to their limit (MPP). In other words, we observe that the converter interfacing power generation with available headroom provide grid-support analogous to standard ac-GFM control. In contrast, the power generation of $\text{PV}_1$ is  approximately constant and resembles an ac-GFL control with maximum power point tracking. In other words, the dual-port GFM control keeps the power output of $\text{PV}_1$ approximately at the MPP. Finally, by mapping power imbalances between ac~1 and ac~2 the VSC-HVDC system autonomously leverages the reserves in ac~1 to provide GFM functions to ac~2. While standard VSC-HVDC controls require assigning ac-GFM/dc-GFL and ac-GFL/dc-GFM functions of the VSCs at the design stage \cite{GO+21}, the proposed dual-port GFM control law inherently achieves the desired behavior without assigning GFM and GFL roles by mapping power imbalances between the areas.

After the setpoint update additional generation is provided by the turbine governor system and dc source in ac~1 while the PV system returns to its nominal operating point (marked by a triangle in Fig.~\ref{fig:pvcharacteristic}).

\section{Conclusion and outlook} \label{sec:conclusion}
In this paper, we proposed a novel grid-forming (GFM) control paradigm for dc/ac voltage source converters that simultaneously imposes the converter ac voltage and controls the dc voltage (dual-port GFM). Conceptually, dual-port GFM control unifies standard functions of GFM and GFL converters used to interface, e.g., renewable generation and HVDC transmission. We developed a graph representation of power system combining ac and dc subgrids and reduced-order linear dynamical models of converters, machines, and power sources that, in abstraction, model a wide range of devices. For this complex class of power systems, we obtained stability conditions that only require partial knowledge of the systems topology. Finally, we used a high-fidelity case study to illustrate the main features uncovered by our theoretical analysis. While these results are encouraging, there is a need for more detailed studies to understand how to leverage the results in applications such as wind turbine control. Moreover, converter current limiting is a crucial aspect that is well understood for ac-GFL control, but requires further study for dual-port GFM control.

\appendix \label{sec:proofs}

\textit{Proof of Proposition~\ref{prop:derivative.expression}:}	Note that $\mc W_\ac$, $M$, $C$, $K_\g$, $T_\g$ and $\tilde{K}_\theta$ are positive definite matrices. By construction $\mc I_{\g,\ac}^\T \mc I_{\g,\ac}+\mc I_{\g,\dc}^\T \mc I_{\g,\dc}$  is an identity matrix and, because $[\mc I_{\g,\ac}^\T,\mc I_{\g,\dc}^\T]^\T$ has full column rank, $\mc I_{\g,\ac}^\T \mc I_{\g,\ac}+\mc I_{\g,\dc}^\T  \tilde{K}_\theta \mc I_{\g,\dc} \succ 0$. Therefore, $V$ is positive-definite. Moreover, the time derivative of $V_\eta$, $V_{\omega}$, $V_v$ along the trajectories of \eqref{eq:full.dynamics.change.of.coordiantes} restricted to $\bar{P}=\mathbbl{0}_{|\bar{\mc N}_\g|}$ are 
\begin{align*}
\ddt V_\eta =& - \eta^\T \mc W_\ac B_\ac^\T\mc I_{\ad}^\T M_p \mc I_{\ad}  B_\ac \mc W_\ac  \eta + \eta^\T \mc W_\ac B_\ac^\T \mc  I_{\ac}^\T \omega\\
&+ \eta^\T  \mc W_\ac B_\ac^\T  \mc I_{\ad}^\T K_\theta \mc I_{\da} v,\\
\ddt V_{\omega} =&-\omega^\T D \omega  - \omega^\T \mc I_{\ac} B_\ac \mc W_\ac \eta + \omega^\T \mc I_{\g,\ac}  P,\\
\ddt V_v=&	- \frac{1}{2}v^\T \left( \tilde{K}_\theta (G+L_\dc) + (G+L_\dc) \tilde{K}_\theta   \right)v  \\
&-v^\T \tilde{K}_\theta \mc I_{\da} ^\T \mc  I_{\ad} B_\ac \mc W_\ac \eta + v^\T \tilde{K}_\theta \mc I_{\g,\dc} P.
\end{align*}
Next, we note that  $\mc I_{\g,\ac}^\T \mc I_{\g,\ac} \mc I_{\g,\ac}^\T\!=\! \mc I_{\g,\ac}^\T$, $\mc I_{\g,\dc}^\T \mc I_{\g,\dc} \mc I_{\g,\dc}^\T \!=\! \mc I_{\g,\dc}^\T$, and $\mc I_{\g,\ac} \mc I_{\g,\dc}^\T\!=\!\mathbbl{0}_{|\mc N_\g| \times |\mc N_\dc \cup \mc N_\cc|}$ and obtain
\begin{align*}
\ddt V_P\!&=\!-\! P^\T\! (\mc I_{\g,\ac}^\T \mc I_{\g,\ac} \!+\!\mc I_{\g,\dc}^\T \tilde{K}_\theta \mc I_{\g,\dc}) (\mc I_{\g,\ac}^\T \omega \!+\! \mc I_{\g,\dc}^\T v\!+\!K^{-1}\!P)
\end{align*}
Adding the derivatives of functions $V_\eta$, $V_{\omega}$, $V_v$ and $V_P$, and using $\tilde{K}_\theta \mc I_\da^\T\!=\!\mc I_\da^\T K_\theta$ to cancel cross terms results in \eqref{eq:LaSalles.fcn.derivative}.
Next, we show that $\ddt{V}$ is negative semi-definite. Because $M_p \succ 0$, $\mc W_\ac B_\ac^\T \mc I_\ad^\T M_p \mc I_\ad B_\ac \mc W_\ac \succeq 0$ holds. Moreover, by definition {$D \succeq 0$}, and {$G \succeq 0$}, and $\tilde{K}_\theta G =G \tilde{K}_\theta$. 
Next, note that $\frac{1}{2} v^\T (\tilde{K}_\theta L_\dc + L_\dc \tilde{K}_\theta )v=\frac{1}{2} \sum_{i=1}^{N_\dc} {v^i}^\T (\tilde{K}^i_\theta L_\dc^i + L_\dc^i\tilde{K}^i_\theta ) v^i$ and for all $i \in \N_{[1, N_\dc]}$, by definition,  $\tilde{K}_\theta^i=\diag \{k_\theta^i\}$. Hence, $\tilde{K}^i_\theta L_\dc^i\! +\! L_\dc^i \tilde{K}^i_\theta =2k_\theta^i L_\dc^i \succeq 0$. Moreover, $\mc I_{\g,\ac}^\T \mc I_{\g,\ac}+\mc I_{\g,\dc}^\T  \tilde{K}_\theta \mc I_{\g,\dc} \succ 0$ is a diagonal matrix and we can conclude that  $\ddt V\leq 0$ holds. \hfill \QED
	
To prove Proposition \ref{prop:max.inv.set}, we introduce vectors $\eta^i$, $\omega^i$, $v_\ad^i$ that correspond to the angle differences, frequencies, and the dc voltages of the devices connected to the $i^\text{th}$ ac subgrid. Analogously to  $\mc I_\ac$ and $\mc I_\ad$ (cf. Sec.~\ref{subsec:power.system.model}), we define matrices $\mc I_\ac^i \in \{ 0,1 \}^{ |\mc N_\ac^i| \times (|\mc N_\ad^i \cup \mc N_\ac^i|)}$ and $\mc I_\ad^i \in \{0,1\}^{ |\mc N_\ad^i| \times (|\mc N_\ad^i \cup \mc N_\ac^i|)}$ that extract the machine and converter variables of the  $i^\text{th}$ ac subgrid, i.e., $\mc I_\ac^i \theta^i$ is the vector of machine angles. Moreover, $\mc I_{\ac^\g}^i$, $\mc I_{\ac^\el}^i$ and $\mc I_{\ac^\oo}^i$ extract variables of machines with power sources with $k_{\g,i}>0$, losses $d_i>0$, and the remaining machines. Similarly, $\mc I_{\ad^\g}^i$, $\mc I_{\ad^\el}^i$ and $\mc I_{\ad^\oo}^i$ extract variables of converters with power sources with $k_{\g,i}>0$, losses $g_i>0$, and the remaining converters. Next, the  following Lemma is needed for the proof of Proposition \ref{prop:max.inv.set}.
\begin{lemma} \label{lemma:full.column.rank}
Let $X_0^i=\mc I_\ad^i L_\ac^i \mc I_\ac^i$ if $|\mc N_\acdc^i| \geq |\mc N^i_\ac| + \mu_{\mc N}$ and 
\begin{align*}
X^i_0 = \begin{bmatrix}
\mc I_{\ac^\el}^i \mc I_\ac^i \\ \mc I_{\ac^\g}^i \mc I_\ac^i \\ \mc I_\ad^i
\end{bmatrix} \!\! L_\ac^i \!\! \begin{bmatrix} \mc I_{\ac^\oo}^i \mc I_\ac^i\\ \mc I_{\ad^\oo}^i  \mc I_\ad^i \end{bmatrix}^\T 
\end{align*}
if  $|\mc N_\acdc^i| + 
\mu_{\mc N} < |\mc N^i_\ac|$. If, for all $i \in \N_{[1,N_\ac]}$ with $\mc N_\ac^i\neq \emptyset$, there exists $K^i \in \N$ such that Algorithm \ref{alg:removal} terminates with $\bar{\mc C}_{K^i}^i =\emptyset$, then $X_0^i$ has full column rank for all systems obtained by deleting up to $\mu_{\mc N}$ nodes and $\mu_{\mc E}$ edges from \eqref{eq:full.dynamics.change.of.coordiantes}.
\end{lemma}
\begin{proof}
For any ac subgrid $i \in \N_{[1,N_\ac]}$, the Laplacian matrix $L^i_\ac$ of the graph $\mc G^i_\ac$ can be partitioned according to the node partition in Definition~\ref{def:partition} to obtain
\begin{align*} \label{eq:general.laplacian.partition]}
L^i_\ac = \begin{bmatrix}
X_{{\mc D}^i, {\mc D}^i}\!\!\!\! & 	X_{{\mc D}^i, {\mc C}^i}\!\!\!\! & 	X_{{\mc D}^i, \mc F^i} \\ 
X_{{\mc C}^i, {\mc D}^i}\!\!\!\! & 	X_{{\mc C}^i, {\mc C}^i}\!\!\!\! & 	X_{{\mc C}^i, \mc F^i} \\
X_{\mc F, {\mc D}}\!\!\!\! & 	X_{\mc F^i, {\mc C}^i}\!\!\!\! & 	X_{\mc F^i, \mc F^i} 
\end{bmatrix}\!,\; X^i_0 \coloneqq \begin{bmatrix}
X_{{\mc D}^i,{\mc C}^i}\!\!\!\! & X_{{\mc D}^i,\mc F^i}\\
X_{\mc F^i,{\mc C}^i}\!\!\!\! & X_{\mc F^i,\mc F^i}
\end{bmatrix}\!.
\end{align*}
At each iteration $k$, Algorithm~\ref{alg:removal} removes a node $\xi_k \in \bar{{\mc C}}^i_k$ if it is connected to $\mu_{\max}+1$ single-edge nodes $l_k \in {\mc D}^i$. Thus, at $k=0$, the only non-zero element in the first $\mu_{\max}+1$ rows is the element $(l_0,\xi_\rho)$, where $\rho\in \N_{[1,\mu_{\max}+1]}$. Using elementary row and column operations, we obtain
\begin{align}\label{eq:matrixiteration}
X_j^i=\begin{bmatrix}
{\diag\{x_{l_k,\xi_k}\}_{k=0}^{j-1}} & \mathbbl{0}_{j\times |{\mc C}^i|-j} & \mathbbl{0}_{j\times |\mc F^i|}  \\
\mathbbl{0}_{(|{\mc D}^i|-j)\times  j} &  X_{{\mc D}_j^i, {\mc C}_j^i}  & X_{{\mc D}_j^i,\mc F^i} \\
\mathbbl{0}_{|\mc F^i|\times  j} &X_{\mc F_j^i,{\mc C}_j^i}  & X_{\mc F^i,\mc F^i}\\
\multicolumn{3}{c}{\mathbbl{0}_{j\mu_{\max}\times (j+|\mc C_j^i \cup \mc F^i|)}}
\end{bmatrix},
\end{align}

with $j=1$. Moreover, the matrix $X_{{\mc D}^i_j, {\mc C}^i_j}$ is obtained by removing the $l^{\text{th}}_{j-1}$ row and $\xi^{\text{th}}_{j-1}$ column from $X_{{\mc D}^i_{j-1},{\mc C}^i_{j-1}}$, the $l^{\text{th}}_{j-1}$ row from $X_{{\mc D}^i_{j-1},\mc F^i}$, and the $\xi^{\text{th}}_{j-1}$ column from $X_{\mc F^i_{j-1},{\mc C}^i_{j-1}}$. Moreover, if at $k=1$, Algorithm~\ref{alg:removal} removes a node $\xi_1 \in \bar{{\mc C}}^i_1$, then the only non-zero element in the $l^\text{th}_1$ row of \eqref{eq:matrixiteration} is the element $(l_1,\xi_1)$ and we can again apply elementary row and column operations to obtain \eqref{eq:matrixiteration} with $j=2$. Induction over $j$ until $j=K=|{\mc C}^i|$ results in
	\begin{align*}
X_K^i \coloneqq \begin{bmatrix} \diag\{x_{l_k,\xi_k}\}_{k=0}^{K-1} & \mathbbl{0}_{K\times|\mc F^i|} \\
\mathbbl{0}_{(|{\mc D}^i|-K)\times  K} & X_{{\mc D}_K^i, \mc F^i } \\
\mathbbl{0}_{|\mc F^i|\times K}   & X_{\mc F^i,\mc F^i}\\
\multicolumn{2}{c}{\mathbbl{0}_{j\mu_{\max}\times (K+|\mc F^i|)}}
 \end{bmatrix}.
\end{align*}
Finally, the matrix $X_{\mc F^i,\mc F^i}$ is obtained by removing $|{\mc D}^i|+|{\mc C}^i|$ rows and columns from the Laplacian matrix $L^i_\ac$ of the graph $\mc G^i_\ac$. Because the graph $\mc G^i_\ac$ is connected, deleting rows and columns of $L^i_\ac$ results in a loopy-Laplacian with at least one diagonal element that is larger than the sum of the absolute values of the other elements in its row. Thus, $X_{\mc F^i,\mc F^i}$ is a loopy Laplacian with full rank (cf. \cite{DSB2018}) and it follows that $\rank \{X^i\}=\rank \{X^i_K \} =|{\mc C}^i|+|\mc F^i|$.
\end{proof}

\textit{Proof of Proposition~\ref{prop:max.inv.set}:} To characterize the largest invariant set contained in $\mc S$, we first separate the time derivative of $V$ (see  \eqref{eq:LaSalles.fcn.derivative}) into terms corresponding to the dc and ac subgrids:
	\begin{align*} \label{eq:lasalle.derivative.rewritten.out}
		\begin{split} 
			&\ddt{V} \!\!=\!-\!\sum_{i=1}^{N_\ac} \! \left(\! (\mc I_{\ad}^i B_\ac^i \mc W_\ac^i \eta^i)^\T \!M_p^i \mc  I^i_{\ad} B_\ac^i \mc W_\ac^i  \eta^i  \!+ \! {\omega^i}^\T \!D^i \omega^i
			\right) \! \\
			&- \! \sum_{j=1}^{N_\dc}  {v^j} ^\T \!  k_\theta^j  (G^j\!+\!L_\dc^j) v^j \! \!  
			- \! P^\T \!(\mc I_{\g,\ac}^\T \mc I_{\g,\ac} \!+\!\mc I_{\g,\dc}^\T \tilde{K}_\theta \mc I_{\g,\dc})K_\g^{-1} \!P.
		\end{split}
	\end{align*}
    Here, $B_\ac^i$ denotes the incidence matrix of the ac subgrid with index $i\in \N_{[1, N_\ac]}$. Moreover, $L_\dc^i$ and $v^i$ correspond to the Laplacian matrix and dc voltages of the dc subgrid with index $i \in \N_{[1,N_\dc]}$. Next, we distinguish the frequencies of the machines with losses $\omega_{\ac^\el}^i \in \R^{|\mc N_{\ac^\el}^i|}$, generation $\omega_{\ac^\g}^i \in \R^{|\mc N_{\ac^\g}^i|}$ and the remaining frequencies $\omega_{\ac^\oo}^i \in \R^{|\mc N_{\ac^\oo}^i|}$. We use $v_{\ad}^i \in \R^{|\mc N_{\ac}^i|}$, $v_{\da}^j \in \R^{|\mc N_{\da}^j|}$, and $v_{\dc}^j \in \R^{|\mc N_{\dc}^j|}$  to denote the dc voltages of the converters in the $i^\text{th}$ ac subgrid, the converter dc voltages in the $j^\text{th}$ dc subgrid, and the dc node voltages in the $j^\text{th}$ dc subgrid. Like the machine frequencies, we separate these vectors and, e.g., use $v_{\ad^\g}^i\in \R^{|\mc N_{\ad^\g}^i|}$ to denote the dc voltages of the converters with generation.
	Moreover, $\ddt V (x) =0$ holds for $x \in \mc S \coloneqq \big(\cup_{i=1}^{N_\ac} \mc A^i\big) \cup \big(\cup_{j=1}^{N_\dc} \mc V^j\big) \cup \mc P$, with
	\begin{align*}
		\mc A^i \!&\coloneqq \left \{x \in \R^n \left | \begin{array}{c} \mc  I_{\ad}^i B_\ac^i \mc W_\ac^i \eta ^i=\mathbbl{0}_{|\mc N_\ad^i |}, \\ \omega_{\ac^\el} ^i =\mathbbl{0}_{|\mc N_{\ac^\el}^i|} \end{array} \right . \right \}\!,  \\ 
		\mc V^j \!&\coloneqq \left \{ x \in \R^n \left | \begin{array}{c} \ L_\dc^j v^j=\mathbbl{0}_{|\mc N_\da^j \cup \mc N_\dc^j|}, \\  v_{\ad^\el}^j =\mathbbl{0}_{|\mc N_{\da^\el}^j|}, \ v_{\dc^\el}^j =\mathbbl{0}_{|\mc N_{\dc^\el}^j|} \end{array}\right . \right \}\!, \\*
		\mc P\!&= \left \{  x \in \R^n \left | \ P=\mathbbl{0}_{|\mc N_\g|} \right . \right \}\!.
	\end{align*}
	Next, we characterize the largest invariant set $\mc M \subseteq \mc S$ to show that $\mc M = \{\mathbbl{0}_n\}$. To this end, we note that $\mc S$ can be rewritten as $\mc S = \{ x \in \R^n \vert S x = \mathbbl{0}_n \}$ and use the fact that $\ddtn{k} S x(t) = \mathbbl{0}_n$ needs to hold for all $x(t) \in \mc M$ and all $k \in \N_0$. 
	Using $K_\theta^i=\diag\{k_{\theta,l}\}_{l=1}^{|\mc N_\ad^i|}$ and $x \in \mc S$ it can be verified that, for all $i \in \N_{[1,N_\ac]}$, invariance of $\mc S$ requires  
	\begin{align}
		\mc I_\ad^i B_\ac^i \mc W_\ac^i \ddt{} \eta^i = \mc I_\ad^i L_\ac^i \begin{bmatrix} \omega^i \\ K_\theta^i v_\ad^i \end{bmatrix}=\mathbbl{0}_{|\mc N_\ad^i|}. \label{eq:first.der.theta}  
	\end{align} 
	Using \eqref{eq:first.der.theta} and $x \in \mc S$ invariance of $\mc S$ requires
	\begin{align}
		\ddt P &= - T_p^{-1} K \begin{bmatrix} \omega^\mathsf{T}_{\ac^\g} & v^\mathsf{T}_{\ad^\g} & v^\mathsf{T}_{\dc^\g} \end{bmatrix}^\mathsf{T}=\mathbbl{0}_{|\mc N_\g|}, \label{eq:first.der.P} 
	\end{align}
	Using the same procedure for all $i \in \N_{[1,N_\ac]}$, it can be verified that invariance of $\mc S$ requires
	\begin{subequations}\label{eq:first.all.eq}
		\begin{align}
			\ddtn{2} \omega^i_{\ac^\el} =  {M_{\ac^\el}^i}^{-1}  \mc I_{\ac^\el}^i \mc I_\ac^i L_\ac^i  \begin{bmatrix} \omega^i \\ K_\theta^i v_\ac^i \end{bmatrix} \!=&\mathbbl{0}_{\mc N_{\ac^\el}^i} \label{eq:second.der.omega.el}, \\
				\!\!\! \mc I_\ad^i B_\ac^i \mc W_\ac^i \ddtn{3} \eta^i \!=\! X_0^i \begin{bmatrix} \omega^i \\ K_\theta^i v_\ad^i \end{bmatrix} =&\!\mathbbl{0}_{|\mc N_\ad^i|} \label{eq:third.der.theta},\\
			\ddtn{3} P \!=\! - T_p^{-1} K \ddtn{2}  \begin{bmatrix} \omega^\mathsf{T}_{\ac^\g} & v^\mathsf{T}_{\ad^\g} & v^\mathsf{T}_{\dc^\g} \end{bmatrix}^\mathsf{T}=&\mathbbl{0}_{|\mc N_\g|} \label{eq:third.der.P}. 
		\end{align}
	\end{subequations}
	Using \eqref{eq:third.der.P}, for all $i \in \N_{[1,N_\ac]}$, invariance of $\mc S$ requires
	\begin{align}
		\ddtn{2} \omega^i_{\ac^\g} &=  {M_{\ac^\g}^i}^{-1}  \mc I_{\ac^\g}^i \mc I_\ac^i L_\ac^i  \begin{bmatrix} \omega^i \\ K_\theta^i v_\ac^i \end{bmatrix} = \mathbbl{0}_{|\mc N_{\ac^\g}^i|}.\label{eq:second.der.omega.g} 
	\end{align}
	If $X_0^i$ has full column rank, it follows that
	\begin{align} \label{eq:equivalent.null.space}
		\Null (X_0^i  {M^i}^{-1} \mc I_\ac^i L_\ac^i)=\Null ( \mc I_\ac^i L_\ac^i),
	\end{align}
	and combining \eqref{eq:equivalent.null.space}, \eqref{eq:third.der.theta}, and  \eqref{eq:first.der.theta}, we have  $( \omega^i,K_\theta^i v_\ad^i ) \in \Null (\mc I_\ad^i L_\ac^i) \cap \Null ( \mc I_\ac^i L_\ac^i)=\Null(L_\ac^i)$. Because $L_\ac^i$ is the Laplacian of a connected graph, it follows that all entries of $( \omega^i,K_\theta^i v_\ad^i )$ are identical. On the other hand, combining \eqref{eq:first.der.P}, \eqref{eq:first.all.eq}, \eqref{eq:second.der.omega.g},  and  $(\eta^i,\omega^i,v^i,P) \in \mc A^i\cup \mc V^i \cup \mc  P$, we have 
	\begin{align*}
		\begin{bmatrix}  \mc I_{\ac^\el}^i \mc I_\ac^i  \\ \mc I_{\ac^\g}^i \mc I_\ac^i \\ \mc I _\ad^i \end{bmatrix} \! L_\ac^i  \begin{bmatrix} \mathbbl{0}_{|\mc N_{\ac^\el}^i \cup \mc N_{\ac^\g}^i|} \\ \omega_{\ac^\oo}^i \\ \mathbbl{0}_{|\mc N_{\ad^\el}^i \cup \mc N_{\ad^\g}^{(i)}|} \\ K_\theta^i v_{\ad^\oo}^i \end{bmatrix} = \mathbbl{0}_{|\mc N_{\ac^\el}^i \cup \mc N_{\ac^\g}^i \cup \mc N_\ad^i|},  
	\end{align*}
	which can be rewritten as
	\begin{align*}
	X_0^i \begin{bmatrix} \omega_{\ac^\oo}^i \\ K_\theta^i v_{\ad^\oo}^i\end{bmatrix}=\mathbbl{0}_{|\mc N_{\ac^\el}^i \cup \mc N_{\ac^\g}^i \cup \mc N_\ad^i|}.
	\end{align*}
    It follows that $(\omega_{\ac^\oo}^i, k_\theta v_{\ad^\oo}^i) = \mathbbl{0}_{|\mc N_{\ac^\oo}^i \cup \mc N_{\ad^\oo}^i|}$ (i.e., all entries are identical) if $X_0^i$ has full column rank. Under the conditions of the Theorem, Lemma \ref{lemma:full.column.rank} ensures that either $\mc I_\ad^i L_\ac^i \mc I_\ac^i$ or $X_0^i$ has full column rank and, for all $i \in \N_{[1,N_\ac]}$, all entries of $( \omega^i,K_\theta^i v_\ad^i )$ are identical when $x \in \mc M$.
	
	Next, we assume that $N_\dc > 0$ and $N_\ac > 0$ and pick any $j \in \N_{[1,N_\dc]}$. Next, $x \in \mc V^j$ implies that $v^j=v^\star \mathbbl{1}_{|\mc N_{\dc}^j|}$ where $v^\star \in \R$. Using condition~\ref{cond:identical.k.theta.gains}, this results in $\omega^i = k^j_\theta v^\star \mathbbl{1}_{|\mc N_\ac^i|} $ for each $\mc N_\ad^i \cap \mc N_\da^j \neq \emptyset$. Moreover, if $N_\dc > 1$ there exists $l \in \N_{[1,N_\dc]}$ such that $\mc N_\ad^i \cap \mc N_\da^l \neq \emptyset$ and  
	$v^l=v^\star k^j_\theta / k^l_\theta \mathbbl{1}_{|\mc N_\da^l \cup \mc N_\dc^l|}$. Because the graph $\mc G$ is connected, it follows by induction over the dc and ac subgrids that all the frequency and dc voltage deviations are proportional to $v^\star$. By Assumption~\ref{assump:onestab}, at least one dc voltage deviation or frequency deviation will be zero if $x \in \mc M$, and therefore $v^\star=0$ and $\mc M = \{\mathbbl{0}_n\}$. 
	
	If $N_\dc =0$, it follows that $N_\ac = 1$ and $\mc I^i_\ad$ is an empty matrix. Thus, we require $X_0^i$ to have full column rank, and using Assumption~\ref{assump:onestab}.\ref{assump:onestab.ac} it directly follows that $\mc M = \{\mathbbl{0}_n\}$. Moreover, if $N_\ac =0$, it follows that $N_\dc = 1$ and it trivially follows from Assumption~\ref{assump:onestab.dc} and $x\in \mc V^i$ that  $\mc M = \{\mathbbl{0}_n\}$.\hfill \QED
	
Note that if $D\succ 0$ and $G\succ 0$ the proof simplifies and only $\ddt \omega=\mathbbl{0}_{|\mc N_\ac|}$ and $\ddt v=\mathbbl{0}_{|\mc N_c \cup \mc N_\dc|}$ are needed to characterize the invariant set $\mc S$. 

\bibliographystyle{IEEEtran}
\bibliography{IEEEabrv,bib_file}

\end{document}